\documentclass[11pt, letterpaper]{amsart}
\usepackage[plainpages=false]{hyperref}
\usepackage{amsfonts,latexsym,rawfonts,amsmath,amssymb,amsthm, mathrsfs, lscape, calligra}
\usepackage{verbatim}
\usepackage{enumerate}
\usepackage{centernot}
\usepackage{mathtools}
\usepackage{tikz}
\usepackage{tikz-cd}

\usepackage[all]{xy}
\usepackage{graphicx,psfrag}

\usepackage{array, tabularx, color}

\usepackage{setspace}

\newtheorem{thm}{Theorem}[section]
\newtheorem{cor}[thm]{Corollary}
\newtheorem{lem}[thm]{Lemma}

\newtheorem{prop}[thm]{Proposition}

\newtheorem{exam}[thm]{Example}

\theoremstyle{remark}
\newtheorem{rmk}[thm]{Remark}

\theoremstyle{definition}
\newtheorem{defi}[thm]{Definition}

\def \G {\mathcal{G}}
\def \C {\mathbb{C}}
\def \E {\mathcal{E}}
\def \F {\mathcal{F}}
\def \O {\mathcal{O}}
\def \P {\mathbb{P}}

\def \Tr {\text{Tr}}

\def \db {\bar \partial}
\DeclareMathOperator \Sing {Sing}
\DeclareMathOperator \Quot {Quot}
\DeclareMathOperator\dVol {dVol}
\DeclareMathOperator\Vol {Vol}

\DeclareMathOperator \Id {Id}

\newcommand{\CBbb}{\mathbb C}

\newcommand{\PBbb}{\mathbb P}
\newcommand{\QBbb}{\mathbb Q}

\newcommand{\ZBbb}{\mathbb Z}

\newcommand{\Ccal}{\mathcal C}

\newcommand{\Ecal}{\mathcal E}
\newcommand{\Fcal}{\mathcal F}
\newcommand{\Gcal}{\mathcal G}
\newcommand{\Hcal}{\mathcal H}

\newcommand{\Lcal}{\mathcal L}

\newcommand{\Ocal}{\mathcal O}

\newcommand{\Qcal}{\mathcal Q}

\newcommand{\Scal}{\mathcal S}
\newcommand{\Tcal}{\mathcal T}

\newcommand{\Xcal}{\mathcal X}

\newcommand{\SL}{\mathsf{SL}}

\newcommand{\PGL}{\mathsf{PGL}}

\DeclareMathOperator{\Hom}{Hom}

\DeclareMathOperator{\id}{id}
\DeclareMathOperator{\ch}{ch}
\DeclareMathOperator{\imag}{Im}

\DeclareMathOperator{\rank}{rk}
\DeclareMathOperator{\codim}{codim}

\DeclareMathOperator{\vol}{vol}

\DeclareMathOperator{\Sym}{Sym}

\DeclareMathOperator{\Spec}{Spec}

\DeclareMathOperator{\PD}{PD}
\DeclareMathOperator{\Td}{Td}
\DeclareMathOperator{\supp}{supp}

\newcommand{\lra}{\longrightarrow}

\newcommand{\Gr}{{\rm Gr}}

\newcommand{\BM}{\mathsf{BM}}

\numberwithin{equation}{section}

\makeatletter
\@namedef{subjclassname@2020}{\textup{2020} Mathematics Subject Classification}
\makeatother

\begin{document}
\title[Singular Donaldson-Uhlenbeck-Yau Theorem and Degeneration]{ A Donaldson-Uhlenbeck-Yau theorem for normal varieties and semistable bundles on degenerating families}

\author[Chen]{Xuemiao Chen}
\address{Department of Pure Mathematics, University of Waterloo, Waterloo, Ontario
Canada N2L 3G1}
\email{x67chen@uwaterloo.ca}

\author[Wentworth]{Richard A. Wentworth}
\thanks{R.W.'s research is supported in part by NSF grants DMS-1906403
and DMS-2204346.}
\address{Department of Mathematics, University of Maryland, College Park,
MD 20742, USA}
\email{raw@umd.edu}


\begin{abstract}
In this paper, we first prove a  Donaldson-Uhlenbeck-Yau theorem over 
       projective normal varieties smooth in codimension two. 
      As a consequence we deduce
      the polystability of (dual)
      tensor products of stable reflexive sheaves,  and we give
      a new proof of
    the Bogomolov-Gieseker inequality,
     along with a precise characterization of
    the case of  equality. This also improves
     several previously known algebro-geometric 
    results on normalized tautological classes. 
    We study the limiting behavior of 
    semistable bundles over a degenerating family of projective normal 
    varieties. In the case of a family of stable
   vector bundles, we study the degeneration of the corresponding HYM connections and these can be characterized from the algebro-geometric perspective. In particular, this proves another version of the singular Donaldson-Uhlenbeck-Yau theorem for the normal projective varieties in the central fiber. As an application, we apply the results to  the degeneration of stable bundles through
    the deformation to projective cones,  and we explain how our results are
    related to the Mehta-Ramanathan restriction theorem. 
\end{abstract}

\maketitle
\thispagestyle{empty}
\bibliographystyle{amsplain}
\allowdisplaybreaks
\tableofcontents

\section{Introduction}
Two fundamental results in the study of stable bundles over compact
K\"ahler and projective algebraic manifolds are
the Donaldson-Uhlenbeck-Yau (DUY) theorem and the Mehta-Ramanathan (MR) 
restriction theorem. The DUY theorem
proves the existence of a Hermitian-Einstein (HE) metric 
on any 
 slope stable 
vector bundle $\E$ over a compact K\"ahler manifold $(Y, \omega)$
\cite{Donaldson:85, UhlenbeckYau:86, Donaldson:87a}.
The associated Chern connections are called Hermitian-Yang-Mills (HYM)
connections.
The result was generalized by Bando and
Siu to  stable reflexive sheaves  via the
introduction of \emph{admissible} metrics \cite{BandoSiu:94}.
Important consequences include the polystability of tensor products of
stable reflexive sheaves,  as well as the Bogomolov-Gieseker inequality and
a characterization of the case of equality in terms of  projectively flat bundles.
It moreover gives equivalent characterizations of reflexive sheaves with 
nef normalized tautological class. 
We refer to \cite{Nakayama:04} for a more detailed discussion of this. 

Suppose now  that $(Y, \omega)$ is a Hodge manifold, i.e. $[\omega]=c_1(L)$ for some
ample line bundle $L\to Y$.
Let $\F$ be a semistable (resp.\ stable) bundle over $Y$. 
The  MR restriction theorem states for $k$ sufficiently large, 
$\F|_{V}$ is semistable (resp.\ stable) for generic $V\in \P(H^0(Y, L^k))$
\cite{MehtaRamanathan:84, Flenner:84}.
This result was instrumental in Donaldson's proof of the DUY theorem in
higher dimensions \cite{Donaldson:87a}. 
We refer to \cite{HuybrechtsLehn:10} for a discussion
of some other important consequences of the restriction theorem. 

The first result of this paper is a  version of the
DUY and Bando-Siu
theorems for stable reflexive sheaves over projective normal varieties $Y$
which are smooth in codimension $2$. A motivation for a generalization  to the singular case comes from  
the corresponding consequences mentioned above, just as the smooth case. 

Next we study the behavior of semistable bundles on degenerating families,
which gives us another version of the DUY theorem. The motivation  here is to give a strategy for an analytic proof of the MR theorem. 
The idea goes as follows. 
Suppose $\F$ is a stable bundle over a smooth projective manifold $(Y,\omega, L)$, 
with $V\subset Y$  as above.
Deform $Y$ to the  projective normal  cone 
(see Section \ref{Section: MR theorem} for details). 
The naive guess is that the HYM connection on $Y$ deforms  to  a HYM connection on the
cone,  which by the symmetry of the cone
should have implications for the (semi)stability of the restriction.
There is price to pay here due to the fact that we have introduced
singularities on the base. 
Nevertheless, 
assuming the restriction theorem holds on the cone, then by a continuity argument 
one can conclude that the restriction $\Fcal\bigr|_V$ 
is semistable (resp.\ stable). 
Working backwards,
in  the case where the restriction theorem actually does hold, 
one can show that the sheaf defined by the limiting HYM connection is
indeed of cone type, in the obvious sense. 
However, there are plenty of examples that tell us that
this is \emph{not} always true, i.e.\  the limiting HYM connection does not have
to be a cone in general. Also,  as in the singular DUY theorem, a technical
difficulty is that we introduced a singularity to the base,  and the known analytic
results about admissible HYM connections on smooth manifolds
are not known to hold across the singular set in general.

Given the discussion above,  a closely related question naturally arises: 
what is the structure of the limiting HYM connection in general?
This leads to the study of the behavior of semistable bundles over a
degenerating family 
of smooth varieties with a normal central fiber. 
It turns out that  an algebraic geometric result can be constructed  
in this setting using known techniques, 
and this can in turn be used to characterize the analytic 
degeneration. In the special case of a deformation to the projective cone,
this picture  is closely related to the MR restriction theorem.
Indeed, we show that  the limiting HYM connection on the cone 
restricts to the HYM connection on $\F|_V$, if the latter is assumed to be  stable.
Thus,  the HYM connection on $\F|_V$ is obtained from  the HYM connection 
on $\F$ by deforming the K\"ahler metric on the base.

\subsection{Main results}
We now state more precisely the results sketched above.

\subsubsection{Singular DUY theorem} \label{sec:assumptions}
In this section, we assume that $(Y, \omega)$ is a normal projective
variety smooth in codimension $2$, with  
$Y\subset \P^N$ and $\omega=\omega_{FS}|_{Y^{reg}}$.  
Slope stability  of reflexive sheaves on $Y$ with respect to 
$\omega$ can be defined as in the nonsingular case,
and this agrees with the algebro-geometric notion of slope stability with
respect to the hyperplane section $D$ on $Y$ (see Section
\ref{sec:chern}).
With this understood, the first main result is the following. 
\begin{thm}[{\sc Singular DUY}]\label{Singular DUY}
Let $\F$ be a stable reflexive sheaf over $(Y, \omega)$. Then there exists an admissible HE metric on $\F$
    that is unique up to a constant scaling.
    Moreover, for any local holomorphic section $s$ of $\F$, 
    $\log^+|s|^2\in W^{1,2}_{loc}\cap L^\infty_{loc}$.
\end{thm}

As a corollary of Theorem \ref{Singular DUY}, we have  (see Corollary
\ref{cor:polystability}),
\begin{cor} 
Given $\F_1$ and $\F_2$ stable reflexive sheaves over
    $(Y, \omega)$, then $(\Sym^{k}\F_1)^{**}, (\wedge^k \F_1)^{**}, (\F_1 \otimes \F_2)^{**}$ are all polystable. 
\end{cor}
As in the case of smooth manifolds, Theorem \ref{Singular DUY} 
 gives rise to an analytic proof of the Bogomolov-Gieseker inequality of
 Miyaoka \cite[Cor.\ 4.7]{Miyaoka:87},  as
well as a characterization of the case of  equality. 
For the definitions of Chern classes and the pairing below, see Section
\ref{sec:chern}.
\begin{cor}[{\sc Bogomolov-Gieseker inequality}] \label{cor:bogomolov}
    Let $[H]$ be the class of a hyperplane in $Y\subset\PBbb^N$. 
 Suppose $\F\to Y$ is a stable reflexive sheaf. Then 
$$
    (2rc_2(\F)-(r-1)c_1(\F)^2) \cdot [H]^{n-2} \geq 0\ .
$$
Moreover, the equality holds if and only if $\F$ is projectively flat outside $\Sing(\F)$. 
\end{cor}

\begin{rmk}
 We emphasize that   the inequality in Corollary
        \ref{cor:bogomolov}
    is well known. For projective varieties it can be  reduced to the
        surface case by using the MR restriction theorem
        \cite{Miyaoka:87}. More recently, it was extended to 
        K\"ahler spaces  \cite{Wu:21}. The novelty in the statement
        above  is the characterization of the case of equality. 
\end{rmk}

In Section \ref{Section: Critical case of BMY}, we use Theorem
\ref{Singular DUY} to  improve results of H\"oring-Peternell
on normalized tautological classes. 

\subsubsection{Degeneration}
Let $p:\Xcal \rightarrow \P^1= \C \cup \{\infty\}$ be a flat family of
projective varieties so that $X_t:=p^{-1}(t)$ is smooth for $t\neq 0$ and
$X_0$ is normal. A family $\{\Ecal_t\}_{t\neq 0}$ of semistable torsion-free sheaves
$\E_t\to X_t$ is called  ``algebraic'' if we can 
find  a coherent sheaf $\Ecal\to \mathcal{X}$, such that
$\Ecal\bigr|_{X_t}\simeq \Ecal_t$. 
We say    $\E\to \Xcal$ is a semistable degeneration of $\{\E_t\}_{t\neq0}$
if furthermore $\Ecal$ is flat over $\PBbb^1$ (cf.\ \cite[p.\ 34]{HuybrechtsLehn:10}), and $\E_0=\E|_{X_0}$ is torsion-free and semistable. Recall that  
for a semistable sheaf, $\Gr(\Ecal_0)$ and $\Ccal(\Ecal_0)$ denote the
associated graded sheaf of the Jordan-H\"older filtration of $\Ecal_0$,  and its
associated codimension $2$ cycle, respectively 
(see Definition \ref{Definition:algebraic blow-up locus}).
On the algebraic side, we prove the following.

\begin{thm}\label{Theorem: algebraic degeneration}
    A semistable degeneration always exists for any algebraic family
    $\{\Ecal_t\}$ where  $\E_t$ are semistable bundles for all $t\neq 0$. Furthermore, for any two 
    semistable degenerations $\E$ and $\E'$, 
$
\Gr(\E_0)^{**}=\Gr(\E'_0)^{**}
$.
If $X_0$ is smooth in codimension $2$,  then we also have $\mathcal{C}(\E_0)=\mathcal{C}(\E'_0)$.
\end{thm} 

Next, assume that $\{\E_t\}$ are stable vector bundles for $t\neq 0$, and let $A_t$ be the corresponding admissible HYM connections. Known analytic results tell  us that
we can take a gauge theoretical limit for sequences $t_i\to 0$. The limiting data consists of two
parts: a smooth HYM connection $A_\infty$ defined outside some analytic subvariety
$\Sigma$  of $X_0$,  and
the so-called blow-up locus $\Sigma_b=\sum_k m_k^{an} \Sigma_k$, which is an
integral linear combination of pure codimension $2$ subvarieties (see Section \ref{analytic side} for more details). Let $\E_\infty$ be the reflexive sheaf defined by $A_\infty$ over $X_0^{reg}$. 

\begin{thm}\label{Theorem: analytic degeneration}
Fix $\E$ to be any degeneration of $\{\E_t\}_t$ where $\E_t$ are stable vector bundles for $t\neq 0$ and $\E_0$ is torsion-free semistable. Then:
\begin{itemize}
\item[(I)] $\E_\infty$ can be extended to be a semistable reflexive sheaf
    $\overline{\E_\infty}$ over $X_0$. Furthermore, for any local section $s$ of $\overline{\E_\infty}$, $\log^+|s|^2 \in W_{loc}^{1,2} \cap L_{loc}^\infty$;
\item[(II)] $\Gr (\overline{\E_\infty})^{**}=\Gr(\E_0)^{**}$. In
    particular, if $X_0$ is smooth in codimension $2$, then $\overline{\E_\infty}=\Gr(\E_0)^{**}$;
\item[(III)] If $X_0$ is smooth in codimension $2$, then $\mathcal{C}(\E_0)=\Sigma_b$.
\end{itemize}
\end{thm}
In particular, this gives the following version of DUY theorem
\begin{cor}[{\sc Singular DUY for degenerations}]
Given $\E$ a semistable degeneration of $\{\E_t\}_t$ where $\E_t$ are stable vector bundles and $\E_0$ is torsion-free semistable. Then there exists an admissible
    HE metric on $\E_0^{**}$.
\end{cor}

\subsubsection{Relation to the Mehta-Ramanathan restriction theorem}
As another  application of the results above, 
in Section \ref{Section: MR theorem}, we 
 study the degeneration of a semistable bundle through
the deformation to projective cones of smooth divisors,  and  we 
show how our results are related to
the MR restriction theorem. In a sense, 
the result obtained here (see Theorem \ref{Theorem: general restriction} and Corollaries \ref{cor5.9})
describes the general picture when the restriction theorem might fail, 
i.e.\ the  nongeneric setting.
Even when the restriction theorem holds for a given stable bundle, it does
not, of course,   apply to the restriction of HYM connections in general. 
Corollary \ref{cor6.10}, however, can be viewed as an analytic version of 
the restriction theorem: the deformation to the  projective cone
 gives a way to interpolate to  the HYM connection of the restriction. 

\begin{rmk}
Another interesting question is whether we can use the HYM connections on the
    projective cones to recover the original HYM connection 
    on $\F$ 
via a limit 
   as the degree increases.
    This is motivated by the observation that  high degree hypersurfaces 
    can be chosen to cover the ambient variety as the degree $\to \infty$. 
If indeed this can be made rigorous it could lead to an analytic proof of the
    Mehta-Ramananthan restriction theorem. We leave this for future study. 
\end{rmk}

\subsection{Sketch of the proofs} \label{sec:sketch}
We first need definitions of Chern classes. 
Since we are working over singular varieties,
 Chern classes are naturally  defined in  corresponding 
homology groups following \cite{BFM:75}. The first Chern class 
can be always defined,  and the second Chern class is defined when the 
base variety is smooth in codimension $2$. 
Corresponding Chern numbers, e.g.\ \emph{slopes}, 
can be computed algebraically. 
The technical difficulty here is to show that the curvature of the
Chern connection for the
admissible Hermitian metric can indeed be used to compute 
Chern numbers as in the Chern-Weil theory of the smooth case. 
This boils down to the fact that the corresponding curvature terms 
indeed define closed currents across the 
singular set, as well as an application of Poincar\'e duality in the ambient smooth manifold (see Proposition \ref{thm:analytic-degree}).
It is here  where
we need to work with normal complex analytic subvarieties of smooth ones with induced metrics.  

Given the above, the idea for the proof of the singular
DUY theorem is to use gauge theoretical methods going back to 
Donaldson \cite{Donaldson:85}, Uhlenbeck-Yau \cite{UhlenbeckYau:86} and Bando-Siu \cite{BandoSiu:94}. We use
Hironaka's theorem on resolution of singularities 
to simultaneously resolve $\F$ and $Y$ and obtain a holomorphic
vector bundle $\widehat{\F}$ on a smooth resolution $\widehat Y\to Y$. 
Since $Y$ is projective, it can be shown quite
easily that $\widehat{\F}$ is stable with respect to a  small perturbation 
$\omega_\epsilon$ of the pullback of $\omega$. 
In particular, by the DUY theorem 
we have a family of HYM  connections 
$A_\epsilon$ on $\widehat{\F}$.  After passing to a subsequence we may find an Uhlenbeck limit $A_\infty$
of $A_\epsilon$ over the regular part $Y^{reg}$ of $Y$. Then one can show that $A_\infty$ defines a reflexive sheaf isomorphic to the original $\F$, so the conclusion follows. This is due to a  theorem of Siu and a use of MR restriction theorem. 

The proof of the existence of a semistable degeneration of semistable bundles is a slight generalization of Langton's results (\cite{Langton:75}) by following the argument in \cite{HuybrechtsLehn:10}. Now we explain about the uniqueness. For two different semistable degenerations, by
restricting to a flat family of high degree curves, one can get the
uniqueness of $\Gr^{**}$. For the blow-up locus, we need to assume $X_0$
is smooth in codimension $2$ so that we can restrict our family to 
a family of high degree smooth surfaces,
and thus get uniqueness of the
algebraic blow-up locus $\mathcal{C}$. This relies crucially on the known
results for the  compactification of semistable torsion-free sheaves over smooth
projective surfaces.

Assume the family above consists of stable torsion-free sheaves
$\Ecal_t$. 
We can take a gauge theoretical limit of the corresponding HYM
connections $A_t$ and get an admissible HYM connection $A_\infty$ and the blow-up 
locus $\Sigma_b$ which is an integer linear combination of pure codimension
$2$ subvarieties of $X_0$. Since in this case, we deal with normal projective varieties in general, the fact that $A_\infty$ defines a semistable reflexive sheaf over $X_0$ is very technical. Assuming this, since we are working in the projective case, we can realize $\E_\infty$ as the double dual of some sheaf represented by an element in a certain Quot scheme which naturally carries the information of the blow-up locus. This transforms our problem into a similar problem which can be dealt with as the uniqueness properties of the semistable degenerations. 

Let $\E_\infty$ be the reflexive sheaf $A_\infty$ defines over $X_0^{reg}$. Now we first explain why $\E_\infty$ can be extended to be a reflexive sheaf over $X_0$.  To solve these technical issues, the projectivity plays a key role 
in that we can construct a
lot of sections of the twisted bundle $\E_\infty(k)$ for $k$ fixed but large. 
Indeed, twisting by the pullback of  the positive line bundle on $Y$,
we know that the bundles $\E_t(k)$ have  many sections. 
We wish to take limits  to produce nontrivial
sections of $\E_\infty(k)$. Using the HYM condition, 
the construction proceeds once one proves that the Sobolev constant of 
$(X_t,\omega_t)$ has a uniform bound which follows from a result of Leon Simon.

Given the above, we want to show that $\E_\infty$ can be extended as  a 
coherent reflexive sheaf over $X_0$. This follows from an induction argument
(cf.\ \cite{Serre:66}).
We start with one limiting section
$s$: it will generate a rank $1$ saturated subsheaf $\mathcal{L}$ of
$\E_\infty$. Furthermore, by the Remmert-Stein theorem
it can be extended to a rank $1$ reflexive sheaf over $Y$. 
Now by increasing $k$, we can find another section that generates a higher 
rank coherent subsheaf of $\E_\infty$,  since by the Riemann-Roch theorem
$\mathcal{L}(k)$ cannot contain all the limiting sections for $k$ large. 
This now gives us a rank $2$ saturated coherent subsheaf of $\E_\infty$
which can also be extended as a coherent sheaf over $X_0$. 
Continuing in this way, the result is obtained. Furthermore, one can show
the limiting sections generate $\E_\infty$ outside some codimension $2$ subvariety. Thus the extension follows, the fact that $\overline{\E_\infty}$ is semistable again follows from certain regularity results, which enable us to do integration by parts.

\subsection*{Acknowledgments} The authors would like to thank Song Sun for helpful comments and suggestions.
They also thank an anonymous reader for pointing out an error in a previous
version of this paper, as well as the referee for making numerous
suggestions that improved the exposition.

\section{Singular DUY Theorem}

\subsection{Admissible HYM connections}
Let $(Y,\omega)$ be a normal subvariety of $\P^N$ with $\omega=\Omega_{FS}|_{Y^{reg}}$. Here we write 
$
Y=Y^{reg} \cup Y^{s}
$,
where $Y^{reg}$ (resp.\ $Y^s$) denotes the smooth (resp.\ singular) locus
of $Y$.
Following \cite{BandoSiu:94}, we will use the following definition.
\begin{defi} \label{def:admissible}
Let $\Fcal\to Y^{reg}\setminus \Sigma$ be a holomorphic vector bundle,
    where $\Sigma$ is 
    a closed set of $Y^{reg}$ of Hausdorff real codimension $\geq 4$. 
     A  Hermitian metric $h$ on $\Fcal$ is called
    \emph{admissible} if 
\begin{itemize}
    \item $\sqrt{-1} \Lambda_{\omega} F_{h}$ is bounded;
\item $\int_{Y} |F_{h}|^2<\infty$, \footnote{The integration 
    is understood  to be taken over the smooth locus  $Y^{reg}$.}
\end{itemize}
    where $F_{h}$ denotes the curvature of the  Chern connection
    of the pair $(\F, h)$ defined on the smooth locus of $\Fcal\to
    Y^{reg}$.
    If furthermore,  
    \begin{equation} \label{eqn:hym}
        \sqrt{-1}\Lambda_{\omega}F_{h}=\mu \Id
    \end{equation}
for some constant $\mu$, then $h$ is called a HE metric and
    the connection $A$ is called an admissible HYM connection. 
\end{defi}

\subsection{Chern classes} \label{sec:chern}
We need definitions of Chern classes of coherent sheaves on singular
varieties. It seems there is not a standard definition, 
but we are only
interested in $\ch_1$ and $\ch_2$, and for these different methods likely
give the same answer with our assumptions. 
Most importantly, we wish to compare these Chern classes with the currents
defined by admissible metrics. 
It will be useful to use the approach of Baum-Fulton-MacPherson
\cite{BFM:75}. 
Recall the notation of the previous section, although here we choose an
arbitrary embedding $Y\subset M$, where $M$ is smooth and projective. 
As will be clear in  the following, we will  define the first Chern class  in general 
and the   second Chern class only when $\dim Y^s \leq n-3$. 
(Co)homology groups will be taken with complex coefficients. 
Let $\F\to Y$ be a coherent analytic sheaf.  
Choose a locally free resolution on $M\supset Y$:
$$
0\lra\F_\ell\lra\cdots\lra\F_0\lra \jmath_\ast\F\lra 0\ ,
$$
and denote the complex of $\F_i$'s by $\F_\bullet$.
Let
$$
\ch^\ast(\Fcal_\bullet):=\sum_{i=0}^\ell (-1)^i\ch^\ast(\Fcal_i)\ ,
$$
where $\ch^\ast(\Fcal_i)\in H^\ast(M)$ is the Chern character. 
Since $\jmath_\ast \Fcal$ is supported on $Y$, $\ch^\ast(\Fcal_\bullet)$
defines a class in $H^\ast(M,M\setminus Y)$ (cf.\ \cite[p.\ 285]{Suwa:00}).
We set: $$\ch^M_Y(\F):=A(\ch^\ast(\F_\bullet))\in H_\ast(Y)\ .$$ 
Here, $A: H^\ast(M,M\setminus Y)\to H_\ast(Y)$ is Alexander duality.
Define:
\begin{equation} \label{eqn:tau}
    \tau(\F)=\Td(M)\cap \ch^M_Y(\F)\ .
\end{equation}
Then $\tau(\F)$ is independent of the embedding $Y\subset M$. 

In the case where $\F$ is the restriction of a vector bundle on $M$, 
then $\tau(\F)=\ch^\ast(\F)\cap \tau(Y)$, where
$\tau(Y):=\tau(\Ocal_Y)$ (\cite[p.\ 116 (4)]{BFM:75}). 
Note that $\tau_n(Y)=[Y]$, the fundamental class of $Y$  (\cite[p.\
129]{BFM:75}).
Motivated by this, we make the following

\begin{defi} \label{def:chern}
    Chern classes $\ch_i(\Fcal)\in H_{2n-2i}(Y)$, $i=0,1,2$,  are defined by
\begin{align}
    \begin{split} \label{eqn:chern-classes}
        \ch_0(\F)&=\tau_n(\F)=\rank(\F)[Y]\ , \\
        \ch_1(\F) &= \tau_{n-1}(\F)-\rank(\F)\cdot \tau_{n-1}(Y)\ . \\
        \ch_2(\F)&=\tau_{n-2}(\F)-\rank(\F)\cdot
    \tau_{n-2}(Y)-\ch_1(\F)\cdot\tau_{n-1}(Y)\ .
    \end{split}
\end{align}
\end{defi}
Let us explain the last term on the right hand side of the expression for
    $\ch_2(\Fcal)$. Since $\codim(Y^s)\geq 3$, 
    $$
    H_{2n-2i}(Y)\simeq H^{\BM}_{2n-2i}(Y^{reg})\ ,\ i=0,1,2\ ,
    $$
    where $H^{\BM}_\ast(Y^{reg})$ denotes the Borel-Moore, or locally
    finite, homology.  Using Poincar\'e duality $H^{\BM}_p(Y^{reg})\simeq
    H^{2n-p}(Y^{reg})$ and the cap product:
    $$
    H^k(Y^{reg})\otimes H_p^{\BM}(Y^{reg})\lra H_{p-k}^{\BM}(Y^{reg})
    $$
    gives  well-defined intersection pairings
    $$
    \ch_1(\Fcal)\cdot\tau_{n-1}(Y)\ , \ \ch_1(\Fcal)\cdot \ch_1(\Fcal) \in
    H_{2n-4}(Y)\ .
    $$
    Another consequence of the assumption $\codim(Y^s)\geq 3$ means we have
    a homomorphism
    \begin{equation} \label{eqn:duality}
        H_{2n-2i}(Y)\simeq H_{2n-2i}^{\BM}(Y^{reg})\lra \left[
            H_c^{2n-2i}(Y^{reg})\right]^\ast\ , \ i=0,1,2\ .
    \end{equation}

If $h$ is an admissible metric
on a reflexive coherent sheaf $\F$
in the sense of Definition \ref{def:admissible}, 
then the Chern-Weil forms $\ch_1(\F,h)$ and $\ch_2(\F,h)$
for the Chern connection of $(\Fcal,h)$
also give elements of $[H^{2n-2i}_c(Y^{reg})]^\ast$ via integration. 
We wish to prove:

\begin{prop} \label{prop:chern}
    For $i=0, 1,2$,
    the image of the classes $\ch_i(\Fcal)$ under the map
    \eqref{eqn:duality} coincide with 
    the classes of $\ch_i(\F,h)$.
\end{prop}

Let $\widehat M\supset\widehat Y\stackrel{\pi}{\lra} Y$ be an embedded
resolution of singularities of $Y$,
$\hat\jmath: \widehat Y\hookrightarrow\widehat M$, and let $\widehat\F=\pi^\ast\F$. 
We furthermore let $\widetilde \F_\bullet=\pi_!(\widehat \F)$. 
Since: 
$$
0\lra \F\lra\pi_\ast(\widehat\F)\lra T\lra 0\ ,
$$
where $\supp(T)\subset Y^s$, it follows that
$\tau_i(\F)=\tau_i(\widetilde\F_\bullet)$ for $i\geq n-2$
if $Y^s$ has codimension $\geq 3$. 
Moreover, by the naturality of the homological Todd class we have:
$\tau(\widetilde\F_\bullet)=\pi_\ast\tau(\widehat\F)$. 
Hence, 

\begin{equation} \label{eqn:tau-updown}
    \tau_i(\F)=\pi_\ast \tau_i(\widehat\F)\ ,\ i\geq n-2\ .
\end{equation}

Let $\widehat \F_\bullet$ be a locally free resolution of $\widehat
\F$ on the smooth variety $\widehat Y$.  Let $\widehat\rho:\widehat U\to
\widehat Y$ be a holomorphic retraction of an open neighborhood $\widehat
U\subset\widehat M$ to $\widehat Y$. 

\begin{lem} \label{lem:tau}
    $\displaystyle
    \tau(\widehat\F)=\ch^\ast(\widehat\rho^\ast\widehat\F_\bullet)\cap
    \tau(\widehat Y)$.
\end{lem}

\begin{proof}
    Since $\widehat Y$ is smooth, we have:
    \begin{align*}
        \tau(\widehat \F)&=\Td(\widehat M)\cap \ch_{\widehat
        Y}^{\widehat M}(\widehat \F) \\
        &=\Td(\widehat M)\cap
        A(\ch^\ast(\widehat\rho^\ast(\widehat\F_\bullet)\cup
        \ch^{\ast}(\hat\jmath_\ast\Ocal_{\widehat Y})) \\
        &=\Td(\widehat M)\cap
        \ch^\ast(\widehat\rho^\ast(\widehat\F_\bullet)\cap
        A(\ch^{\ast}(\hat\jmath_\ast\Ocal_{\widehat Y})) \\
        &=\ch^\ast(\widehat\rho^\ast(\widehat\F_\bullet))\cap
        (\Td(\widehat M)\cap
        \ch_{\widehat Y}^{\widehat M}(\Ocal_{\widehat Y})) \\
        &=\ch^\ast(\widehat\rho^\ast(\widehat\F_\bullet))\cap
        \tau(\widehat Y) \ .
    \end{align*}
    The first to second (resp.\ second to third) lines is in \cite[p.\ 291
    (resp.\ p.\ 288)]{Suwa:00}.
\end{proof}

\begin{proof}[Proof of Proposition \ref{prop:chern}]
    Let $\Omega$ be a smooth, compactly supported, closed form on $Y^{reg}$
    of degree $2n-2i$, $i=0,1,2$, representing a class in
    $H^{2n-2i}_c(Y^{reg})$.  
We may assume that $\supp\Omega$ is contained in the smooth locus of
    $\F$. Choose connections $A_\bullet$ on the bundles in the
    resolution $\widehat \rho^\ast(\widehat \F_\bullet)$, and let 
    $\ch^\ast(\widehat\rho^\ast\widehat\F_\bullet, A_\bullet)$ denote
    the alternating sums of Chern-Weil forms. Then on the support of $\Omega$ (or
    $\pi^\ast\Omega$), 
    $\ch^{n-i}(\widehat\rho^\ast\widehat\F_\bullet, A_\bullet)$ and
    $\ch_{n-i}(\F,h)$ differ by a smooth exact form.
    We will denote the images of classes by the map \eqref{eqn:duality} with the same
    notation.
    Now, on the one hand, from \eqref{eqn:tau-updown}, we have
    $$
     \tau_i(\F)([\Omega])=
    \pi_\ast\tau_i(\widehat\F)([\Omega])=
    \tau_i(\widehat\F)(\pi^\ast[\Omega])=\pi^\ast[\Omega]\cap \tau_i(\widehat\F)\ .
    $$
    On the other hand, from Lemma \ref{lem:tau},
    \begin{align*}
        \pi^\ast[\Omega]\cap \tau_i(\widehat\F)&=\pi^\ast[\Omega]\cap\left(
\ch^\ast(\widehat\rho^\ast\widehat\F_\bullet)\cap
        \tau(\widehat Y)\right)_i \\
        &=\sum_{j=i}^{n}
        \left(\pi^\ast[\Omega]\cup
        \ch^{j-i}(\widehat\rho^\ast\widehat\F_\bullet)\right)\cap
        \tau_j(\widehat Y) \\
        &=\sum_{j=i}^{n}
        \left[\pi^\ast\Omega\wedge
        \ch^{j-i}(\widehat\rho^\ast\widehat\F_\bullet, A_\bullet)\right]\cap
        \tau_j(\widehat Y) \ .
    \end{align*}
    Since $\tau_n(\widehat Y)=[\widehat Y]$, 
    we have for each case:
    \begin{align*}
        \tau_n(\F)([\Omega])&= \rank(\Fcal)\int_Y \Omega \\
        \tau_{n-1}(\F)([\Omega])&=
        \rank(\Fcal)\tau_{n-1}(Y)([\Omega])
        +\int_Y \Omega\wedge \ch_1(\Fcal, h)
        \\
        \tau_{n-2}(\F)([\Omega])&=\rank(\Fcal)\tau_{n-2}(Y)([\Omega])
        +\tau_{n-1}(Y)([\Omega\wedge \ch_1(\Fcal, h)]) \\
        &\qquad +
        \int_Y \Omega\wedge \ch_2(\Fcal, h)
    \end{align*}
    Comparing with Definition \ref{def:chern}, 
    this completes the proof.
\end{proof}


We also note the following.

\begin{lem} \label{lem:simpson}
    The forms $\ch_1(\F,h)$, $\ch_1^2(\F,h)$, and $\ch_2(\F,h)$
    define closed currents on $Y$ as elements  
    of the dual space of the restriction to $Y$ of smooth forms on $M$.
\end{lem}

\begin{proof}
    Since the curvature is in $L^2$, the integrals are well defined. The
    fact that the currents are closed follows exactly as in \cite[Sec.\
    6.5.3]{Simpson:87}. 
\end{proof}

To define (semi)stability of $\F$ we need a notion of degree.
Using Lemma \ref{lem:simpson}, this can be defined analytically in the usual way.
In the
projective case, let
$[\alpha]\in A_1(Y)$ be the cycle defined by intersecting hyperplanes in
$M\subset\PBbb^N$. 
By choosing a representative in general position, $[\alpha]$ uniquely
defines a homology class in $H_2(Y^{reg})\simeq H^{2n-2}_c(Y^{reg})$. 
Via \eqref{eqn:duality}, 
 we define:
\begin{equation} \label{eqn:degree}
    \deg_\alpha(\F):=\ch_1(\F)\cdot[\alpha] \in \QBbb \ .
\end{equation}
Similarly, we define the slope of $\F$ to be
$\mu_\alpha(\F)=\deg_\alpha(\F)/\rank(\F)$.
Then a reflexive sheaf $\Fcal$ 
is \emph{stable} (resp.\ \emph{semistable}) 
if $\mu_\alpha(\Scal)<\mu_\alpha(\Fcal)$ (resp.\ $\leq$)
for all coherent subsheaves $\Scal\subset\Fcal$ with
$0<\rank\Scal<\rank\Fcal$.

\begin{prop} \label{thm:analytic-degree}
    Let $\omega=\left.\omega_{FS}\right|_{Y^{reg}}$. 
The following holds:
$$
    \deg_\alpha(\F)=\int_Y \ch_1(\F,h)\wedge \omega^{n-1}
    \ .
$$
\end{prop}

\begin{proof}
It follows from Proposition \ref{prop:chern} that for general representatives
    $\alpha$,
$$
    \deg_\alpha(\F)=\int_\alpha \ch_1(\F,h)
    \ .
$$
    We write $\alpha=H_1\cap\cdots\cap H_{n-1}\cap Y$, where $H_i$ are
    generic smooth
    divisors in the class of $\Ocal_M(1)$ on 
    $M$, and $\alpha$ is in the smooth locus of $\F$.
    Choose a  Poincar\'e dual $\PD(H_1)$  supported in a small
    neighborhood of the $H_1$.
    Write: $\omega-\PD(H_1)=dB$, where $B$ is smooth.
    Then by Lemma \ref{lem:simpson}, 
    $$\int_Y\ch_1(\F,h)\wedge d(B\wedge \omega^{n-2})=0$$
    and
    $$
        \int_Y \ch_1(\F,h)\wedge\omega^{n-1}=\int_Y\ch_1(\F,h)\wedge
        \PD(H_1)\wedge \omega^{n-2}
        \ .
        $$
Continuing in this way, we have
    $$
        \int_Y \ch_1(\F,h)\wedge\omega^{n-1}=\int_Y\ch_1(\F,h)\wedge
        \PD(H_1)\wedge \cdots\wedge\PD(H_{n-1})
        \ .
        $$
Since we have assumed that 
        $\PD(H_1)\wedge \cdots\wedge\PD(H_{n-1})$ is supported in the
        smooth locus of $\F$, the usual Poincar\'e-Lelong formula gives
    \begin{align*}&\int_Y\ch_1(\F,h)\wedge \PD(H_1)\wedge \cdots\wedge\PD(H_{n-1})\\
        =&\int_{H_1\cap Y}\ch_1(\F,h)\wedge \PD(H_2)\wedge
        \cdots\wedge\PD(H_{n-1}) \\
        =&\int_{H_1\cap H_2\cap Y}\ch_1(\F,h)\wedge \PD(H_3)\wedge
        \cdots\wedge\PD(H_{n-1}) \\
        &\vdots \\
        =&\int_\alpha\ch_1(\F,h)
        \ .
    \end{align*}

\end{proof}

\begin{cor}\label{cor2.7}
   If $h$ is an admissible HE metric on $\F$, i.e.\ 
    a solution to \eqref{eqn:hym}, 
  then 
    $$\mu=\frac{2\pi}{\rank\F\cdot \vol(Y)} \int_Y \ch_1(\F,h)\wedge
    \frac{\omega^{n-1}}{(n-1)!}=\frac{2\pi}{(n-1)!\vol(Y)}\mu_\alpha(\F)\ .$$
\end{cor}
Furthermore, by  \cite[Ex.\ 18.3.6]{Fulton:84}, we have the following,
which explains why we have defined the Chern classes in the way we did.
\begin{cor}[{\sc Asymptotic Riemann-Roch formula}]\label{ARR}
The Hilbert polynomial of $\F$ over $(Y, \O_Y(1))$ is given by 
$$P_{\F}(k)=\rank(\F)(a_1 k^n+ a_2 (\mu(\F)+ \deg(Y))k^{n-1})+O(k^{n-2})$$
for some universal constants $a_1, a_2$.
\end{cor}

\subsection{Singular Bott-Chern formula}
In the following, we fix any admissible metric on $\O_Y(-k)$. Assuming $\dim
Y^s \leq n-3$, the argument in \cite{SibleyWentworth:15} gives the
following singular Bott-Chern formula based on our definition of $\ch_2$.
Recall that associated to a torsion-free sheaf $\Fcal$ is a cycle
$\Ccal(\Fcal)$ consisting of the irreducible codimension $2$ pieces of the
support of $\Fcal^{\ast\ast}/\Fcal$, with multiplicities. 
\begin{prop}[{\sc Singular Bott-Chern formula}]\label{SingularBottChern}
Assume that $\F$ is a torsion-free sheaf which admits an admissible HYM
    connection $A_\infty$ over $Y$ where $\dim Y^s\leq n-3$. Given an exact
    sequence
    \begin{equation} \label{eqn:phi}
        0\lra\ker{\phi} \lra\O_Y(-k)^{\oplus N}
        \xrightarrow[\hspace*{.75cm}]{\phi} \F \lra 0\ ,
    \end{equation}
then 
$$ 
\ch_2(A_\infty)+\ch_2(\ker\phi)-\ch_2(\O(-k)^{\oplus N}) =
    \Ccal(\Fcal)
$$
    in $[H_c^{2n-4}(Y^{reg})]^\ast$. 
\end{prop}

\begin{proof}
    From the exact sequences \eqref{eqn:phi} and
    $$
    0\lra\F\lra\F^{\ast\ast}\lra T\lra 0
    $$
    and the naturality of the homological Todd class, we have
   $$ 
        \tau(\F^{\ast\ast})= \tau(\F)+\tau(T) 
        =\tau(\Ocal_Y(-k)^{\oplus N})-\tau(\ker\phi)+\tau(T)
        \ .
   $$ 
    By Proposition \ref{prop:chern}, $\ch_2(A_\infty)=\ch_2(\F^{\ast\ast})$, 
    and since $T$ is supported in codimension $2$, 
    $\tau_{n-1}(T)=0$, and 
    $
    \tau_{n-2}(T)=\Ccal(\Fcal)
    $
    (cf.\ \cite[Remark, p.\ 17]{BFM:75}).
    The result follows. 
\end{proof}

\subsection{Proof of Singular DUY theorem}
By Hironaka's resolution of singularities, we may choose a projective
resolution of singularities
$p: \widehat{\PBbb}^N \rightarrow \PBbb^N$ 
with the strict transform  $\widehat{Y}$ of $Y$ being smooth.
We may also assume  that $\widehat\F=((p|_{\widehat Y})^\ast \F)^{**}$
is locally free.
Let $\theta$ be any fixed K\"ahler metric on $\widehat {\PBbb}^N$ 
and $\omega_\epsilon=(p^*\omega_{FS}+\epsilon \theta)|_{\widehat Y}$. 
We sometimes use $\widehat{Y}_\epsilon$  to emphasize that  $\widehat Y$
is endowed with the K\"ahler metric $\omega_\epsilon$.
We have the following elementary observation. 
\begin{lem}\label{lem2.12}
$\displaystyle \mu_\alpha(\F)=\lim_{\epsilon \rightarrow 0} \mu_\epsilon(\widehat{\F})$.
\end{lem}

\begin{proof}
    As in Section \ref{sec:chern}, $\ch_1(\Fcal)$ defines an element of
    $H^2(Y^{reg})$. Then 
 $\ch_1(\widehat{\F})=p^* \ch_1(\F)+\sum_i a_i E_i$, 
    where $E_i$ are the exceptional divisors of the blow-up maps. The
    result follows from the definition \eqref{eqn:degree} of degree.
\end{proof}

\begin{prop}
$(\widehat\F, \omega_\epsilon)$ is stable for $0<\epsilon\ll 1$.  
\end{prop}

\begin{proof}
Write 
$$\omega_\epsilon^{n-1}=\omega^{n-1}+\sum_i c_i^{\epsilon} \alpha_i\ ,$$ 
    where $\alpha_i$ is the wedge product of $(p^\ast\omega_{FS}|_{\widehat{Y}},
    \theta|_{\widehat Y}$  and $c^\epsilon_i$ are
    constants depending on $\epsilon$ and $1\leq i\leq n-1$. Furthermore, $c_i^\epsilon
    \rightarrow 0$ as $\epsilon\rightarrow 0$. To prove stability of $\widehat \F$ we 
    need to show that for any nontrivial proper subsheaf $\widehat{\mathcal S} \subset \widehat{\F}$, 
$
\mu_\epsilon(\widehat{\mathcal{S}})<\mu_\epsilon(\widehat{\F})
$,
where $\mu_\epsilon$ denotes the slope of a sheaf with respect to the
    K\"ahler class $[\omega_\epsilon]$. Let
    $\mathcal{S}=p_*\widehat{\mathcal{S}}$. We have
$$
    \mu_\epsilon(\widehat{\mathcal{S}})=\frac{1}{\rank(\widehat{\mathcal S})}
    \int_{\widehat Y} c_1(\det{\widehat{\mathcal S}}) \wedge
    \omega_\epsilon^{n-1}=\mu(\mathcal S)+ \frac{1}{\rank(\widehat{\mathcal S})}
    \sum_i c_i^{\epsilon}\int_{\widehat Y} c_1(\det{\widehat{\mathcal S}}) \wedge \alpha_i
    \ .
$$
Let
$f(\epsilon, \widehat{\mathcal{S}})=\mu_\epsilon(\widehat \F)-\mu_\epsilon(\widehat{\mathcal S})
$. 
Then
    \begin{align*}
        f(\epsilon, \widehat{\mathcal{S}})&=\mu(\F)-\mu(\mathcal S)\\
        &\ + 
    \sum_i c_i^{\epsilon}\left(\frac{1}{\rank(\widehat{\mathcal F})}
    \int_{\widehat Y} c_1(\det{\widehat {\mathcal F}})
    \wedge \alpha_i-\frac{1}{\rank(\widehat{\mathcal{S}})}\int_{\widehat Y}
    c_1(\det{\widehat{\mathcal{S}}}) \wedge
    \alpha_i\right) \ .
    \end{align*}
It suffices to show  that for
$\epsilon$ small enough, one has   
$f(\epsilon, \widehat{\mathcal{S}})>0$ for any such
$\widehat{\mathcal{S}}$ as above. 
In order to prove this, we first note that there exists $C>0$ so that 
$\mu(\F)-\mu(\mathcal S)\geq C>0$, 
for any $\mathcal S=p_*\widehat{\mathcal{S}}$, 
    since $\F$ is stable over $(Y, \omega)$. 
So it suffices to show that there exists a constant $C'$ so that 
$$
    \frac{1}{\rank(\widehat \Scal)}\int_{\widehat Y}  c_1(\det{\widehat\Scal}) \wedge
    \alpha_i\leq C'\ .
$$
To show this, take a line bundle $\mathcal L$ over $\widehat Y$ that is ample enough so that we have an inclusion as 
$
\widehat{\F} \hookrightarrow \mathcal{L}^{\oplus l}
$,
for some $l>0$. In particular, $\widehat\Scal$ can be viewed as a
    subsheaf of $\mathcal{L}^{\oplus l}$, 
    which implies $\det(\widehat \Scal)$ can be viewed as a subsheaf of 
    $\wedge^{\rank (\widehat\Scal)}\mathcal (\mathcal{L}^{\oplus l})$. In particular, since $\int_D\alpha_i\geq 0$ for any divisor $D$, we have 
$$
    \frac{1}{\rank(\widehat \Scal)}\int_{\widehat Y}  c_1(\det{\widehat\Scal}) \wedge
    \alpha_i\leq \int_{\widehat Y} c_1(\mathcal{L}) \wedge \alpha_i\ .
$$
The conclusion follows.
\end{proof}

Given the above, the idea is to take limits of the HYM connections $A_\epsilon$ on 
$\widehat \F$ over $(\widehat Y, \omega_\epsilon)$. We need the following 
\begin{lem}
There exist $C$ independent of $\epsilon$ so that 
$$
\int_{\widehat Y} |F_{A_\epsilon}|^2 \dVol_\epsilon \leq C\ .
$$
\end{lem}

\begin{proof}
Indeed, we have
$$
\int_{\widehat Y} |F_{A_\epsilon}|^2 \dVol_\epsilon  \leq \int_{\widehat
    Y}\Tr(F_{A_\epsilon} \wedge F_{A_\epsilon})\wedge
    \frac{\omega_{\epsilon}^{n-2}}{(n-2)!}+c\mu_\epsilon^2
    \Vol(\widehat{Y}_\epsilon)\ ,
$$
where $c$ is some dimensional constant. Given this, the result follows from Lemma \ref{lem2.12}.
\end{proof}

Fix any sequence $A_i:=A_{\epsilon_i}$ with $\epsilon_i \rightarrow 0^{+}$
as $i\rightarrow \infty$. Set
$$
\Sigma=\{x\in Y^{reg}: \lim_{r\to 0}\liminf_i r^{4-2n}\int_{B_x(r)}
|F_{A_i}|^2 \dVol_{\epsilon_i} > 0\}\ ,
$$
which is a complex subvariety of $Y^{reg}$ of codimension at least $2$ (see \cite[Thm.\ 4.3.3]{Tian:00}) and can be decomposed as 
$$
\Sigma=\Sigma_b\cup \Sing(A_\infty)
$$
where $\Sigma_b$ denotes the pure codimension $2$ part of $\Sigma$. By passing to a subsequence we can assume
\begin{itemize}
\item as a sequence of Radon measures over $\widehat Y$, 
$$
\mu_i:=|F_{A_i}|^2\dVol_i \rightharpoonup \mu_\infty=|F_{A_\infty}|^2 \dVol+\nu
\ ,
$$
where $\nu|_{Y^{reg}}=\sum_k 8\pi^2 m_k \Sigma_k$. Here $\Sigma_k$ denotes
        the pure codimension $2$ components of $\Sigma$ and $m_k\in \mathbb{Z}$  is usually called the analytic multiplicity of $\Sigma_k$. 
 
\item outside $\Sigma$, up to gauge transformation, $A_i$ converge to $A_\infty$ locally smoothly.
\end{itemize}
Using Siu's extension theorem \cite{Siu:69}, combined with \cite[Thm.\
2]{BandoSiu:94}, we have the following results regarding the extension
property.
\begin{prop}\label{prop3.2}
Let $\Fcal_\infty$ be the holomorphic vector bundle on $Y^{reg}\setminus \Sigma$ defined by $\db_{A_\infty}$.  Then $\iota_*\F_\infty$ is a reflexive coherent sheaf on $Y$ where $\iota: Y^{reg}\setminus \Sigma \rightarrow Y$ denotes the natural inclusion map. 
Furthermore, for any local section $s$ of $\F_\infty$, $\log^{+}|s|^2 \in W^{1,2}_{loc} \cap L^\infty_{loc}$. 
\end{prop}
\begin{proof}
The extension can be done in two steps: 
\begin{enumerate}
\item by Theorem $2$ in \cite{BandoSiu:94}, $F$ can be extended to be a
    coherent reflexive sheaf $\F_\infty\to Y^{reg}$;

\item by  \cite[Thm.\ 5]{Siu:69}, 
    $\F_\infty$ can be further extended as a coherent
        reflexive sheaf over $Y$; this is because  
    $Y^s$ has codimension  $\geq 3$ and $Y$ is normal.
\end{enumerate}

Now we prove the second part regarding local sections. For this, one can directly repeat the 
    argument in
     \cite[Thm.\ 2]{BandoSiu:94} with known results from \cite{Simon:83}.
     Let $y\in Y^s$.  By choosing a local coordinate of $\P^N$ at $y$, we
     can assume $Y\subset B^{l} \times B^{N-l}$ and $p^{-1}(t)$ is a smooth
     surface for generic $t\in p(Y)$.  Here $y=(0,0)$ and $p: B^l\times
     B^{N-l}\rightarrow B^{N-l}$ denotes the natural projection.
     Furthermore, we can assume the metric on $B^l\times B^{N-l}$ is flat.
     In the following, we will use $A \lesssim B$ (resp. $A\gtrsim B$) to
     denote $A\leq cB$ (resp. $A\geq cB$) for some constant $c$. For
     generic $t\in p(Y)$ as above, denote $Y_t=p^{-1}(t)\cap Y$,
     $f_t=\log^{+}|s|^2|_{Y_t}$ and $F_t=F_{A_\infty}|_{Y_t}$.  We know
     that
    $$
    \Delta f_t \geq -|F_t|\ .
    $$
    Let $\chi$ be a cut-off function in $B^l$ supported in a small
    neighborhood of $0$.     Now we can multiply the above inequality by $\chi^2 f_t$ and do integration by parts to get 
    $$
    \int_{Y_t} |\nabla (\chi f_t)|^2 \leq \epsilon \int_{Y_t} |\chi f_t|^2 + \epsilon^{-1} \int_{Y_t} \chi^2 |F_t|^2 + \int_{Y_t} |\nabla \chi|^2 f_t^2  
    $$
    for any $0<\epsilon\ll 1$.
  By the Sobolev inequality in \cite[Thm.\ 18.6]{Simon:83}, we could conclude 
    $$
    \int_{Y_t} |\nabla (\chi f_t)|^2 \lesssim  \int_{Y_t} |\nabla \chi|^2 f_t^2 + \int_{Y_t} \chi^2 |F_t|^2. 
    $$
    From this, we could integrate it with respect to $t$, and have that for any $K\subset p(Y)$ compact 
    $$
    \int_{p^{-1}(K)} |\nabla' (\chi f_t)|^2 \lesssim \int_{p^{-1}(K)} |\nabla
     \chi|^2 f_t^2+ \int_{p^{-1}(K)} \chi^2 |F_t|^2 
    $$
where $\nabla'$ denotes the total derivatives in the fiber direction of the
    projection $p$. Now we can take finitely many such projections $p$ to
    cover all the derivatives, thus could conclude $\log^{+}|s|^2\in
    W^{1,2}_{loc}$.    Suppose $s$ is a local section of
    $\overline{\E_\infty}$. By \cite[eq.\ (5.5)]{LiTian1995}, we have a Sobolev inequality for any function in $W^{1,2}(Y^{reg})$. Since $\Delta \log^+|s|^2 \geq -\mu(\E_\infty)$, by Moser iteration applied to $\eta^2 \log^{+}|s|^2$ where $\eta$ is a local cut-off function near the point in $\PBbb^N$, we know $\log^+|s|^2\in L^\infty_{loc}$. 
\end{proof}

\begin{lem}\label{Lemma: Extension of bubbling set}
The closure of $\Sigma$ is a subvariety of $Y$.
\end{lem}
\begin{proof} Using the Bishop extension theorem \cite{Bishop:64}, to show $\Sigma_b$ can be extended,
it suffices to show that $\Sigma_k\subset M$ has bounded volume. Indeed, since $\mu_i$ weakly converges to $\mu_\infty$ as a sequence of Radon measures over $\widehat Y$, we have 
$$
\Vol(\Sigma_k) \leq \limsup_i \int_{\widehat Y_{\epsilon_i}} |F_{A_i}|^2 \leq C
\ .
$$
The extension of $\Sing(A_\infty)$ follows from that $\Fcal_\infty$ can be
    extended to be global reflexive sheaf over $Y$ and $\Sing(A_\infty)=\Sing(\Fcal_\infty)$. 
\end{proof}

Now we prove the existence of admissible HE metric in the rank $1$ case.

\begin{cor} \label{cor:rank-one}
Suppose $\rank \F=1$. Then there exists an admissible HYM metric on $\F$.
\end{cor}

\begin{proof}
Indeed, by choosing any higher curve $C\subset Y^{reg}$ so that $A_i|_C$ converges to $A_\infty$ smoothly. Then it is easy to see that $\overline{\F_\infty}|_C$ is isomorphic  to $\Fcal|_C$. Since $C$ can be chosen to be any higher degree, we know $\Fcal$ and $\overline{\Fcal_\infty}$ must be isomorphic.
\end{proof}

\begin{cor}\label{cor:polystable}
$\overline{\F_\infty}$ is semistable. 
\end{cor}

\begin{proof}
Suppose $\Scal\subset \overline{\Fcal_\infty}$ is a proper saturated destabilizing subsheaf with
    $\rank\Scal=m$ and $\Lcal=\det\Scal$. 
Then  the sheaf $(\Lambda^{m}\overline{\F}_\infty\otimes \mathcal{L}^*)^{**}$ admits a global nontrivial section $s$ and 
$$\mu:=\mu((\Lambda^{m}\overline{\F}_\infty\otimes \mathcal{L}^*)^{**})<0\ .$$ 
    By Corollary \ref{cor:rank-one}, $\Lcal$ admits a HE metric. 
Applying Proposition \ref{prop3.2} to
    $(\Lambda^{m}\overline{\F}_\infty\otimes \mathcal{L}^*)^{**}$, we know
    $s$ is bounded. Moreover, we have 
$$
\Delta |s|^2=|\nabla s|^2-\mu |s|^2\geq 0\ .
$$
A contradiction to $s\neq 0$ then follows by  integration against a cut-off near the
    singularities $Y^{s}$. 
\end{proof}

 \begin{proof}[Proof of Theorem \ref{Singular DUY}]
Since $\overline{\F_\infty}$ is semistable and $\F$ is stable, by the
     Mehta-Ramanathan restriction theorem, we can choose a high degree
     curve $C\subset Y\setminus (Y^s \cup \Sigma)$ so that
     $\overline{\F_\infty}|_C$ is semistable and $\F|_{C}$ is stable. By
     assumption, we know $A_{i}|_{C}$ converges to $A|_C$ smoothly up to
     gauge. By the well-known results regarding moduli space of semistable
     bundles over curves, we know $\F|_{C}$ and $\overline{\F_\infty}|_{C}$
     must be $S$-equivalent, thus isomorphic. Since $C$ can be chosen to be
     of any high degree, we conclude that $\F$ and $\overline{\F_\infty}$ must be isomorphic
     (cf.\ \cite[Lemma 5.4]{GrebToma:17}). For uniqueness, suppose $h, h'$
     are two admissible HE metrics on $\Fcal$. Then,  as in Corollary
     \ref{cor:polystable}, $\id\in \Hom(\Fcal, \Fcal)$ is parallel with
     respect to the natural metric $h^\ast \otimes h'$. In particular, if we
     write $h'=h(g,)$, then $g$ is a holomorphic endomorphism of $\Fcal$.
    By stability, it 
     therefore must be a  constant multiple of the identity, and  
     the conclusion follows.
\end{proof}

As a direct corollary of this, combined with the regularity results, we have  
\begin{cor}\label{cor:polystability}
If a reflexive sheaf  $\Fcal\to Y$ admits an admissible HE metric, then
    $\Fcal$ is polystable.
\end{cor}

\begin{proof}
    Semistability follows as in the proof of Corollary
    \ref{cor:polystable}. Suppose $\Scal\subset\Fcal$ is a proper stable saturated
    subsheaf with $\mu(\Scal)=\mu(\Fcal)$. Note that $\Scal$ is reflexive. Let $h'$ be the admissible HE metric on $\Scal$. Now as Corollary \ref{cor:polystable}, we know that $\id: S' \rightarrow \Fcal$ is parallel. This implies there is a holomorphic,
    \cite[Prop.\ 4.1.7]{Kobayashi:87}) implies there is a holomorphic,
    orthogonal splitting $\Fcal=\Scal\oplus \Scal^\perp$ off a set of
    codimension $\geq 3$ in $Y^{reg}$. As above, the sheaves $\Scal$ and
    $\Scal^\perp$ extend as reflexive sheaves on $Y$, and the induced
    metrics are admissible HE. The result now follows by induction on the
    rank.   
\end{proof}

\section{Bogomolov-Gieseker equality and nef tautological
classes} \label{Section: Critical case of BMY}
We begin this section with the
\begin{proof}[Proof of Corollary \ref{cor:bogomolov}]
    Let $Y^{s}\subset S\subset Y$ be a closed subvariety of codimension
    at least $3$ so that $\Fcal$ is locally free on $Y\setminus S$. 
    Using Proposition \ref{prop:chern},   it follows exactly as in the proof of Proposition
    \ref{thm:analytic-degree} that
    \begin{align*} 
        (2rc_2(\Fcal)-(r-1)c_1^2(\Fcal))\cdot& [H]^{n-2} \\ 
        =&\int_Y(2rc_2(\Fcal,h)-(r-1)c_1^2(\Fcal,h))\wedge \omega^{n-2} \\
        =&\int_{Y\setminus S}(2rc_2(\Fcal,h)-(r-1)c_1^2(\Fcal,h))\wedge \omega^{n-2}
\ .
    \end{align*} 
    The inequality then follows as in the smooth case (cf.\ \cite[Thm.\ 4.4.7]{Kobayashi:87}).
    Moreover, in the case of equality, 
     $\Fcal\bigr|_{Y\setminus S}$ is projectively flat.
     Hence, there is a representation 
    $\rho : \pi_1(Y\setminus S)\to\PGL_r(\CBbb)$.
    Since by the codimension assumption, 
    $\pi_1(Y\setminus S)\simeq \pi_1(Y^{reg})$,
    we obtain from $\rho$ a $\PGL_r(\CBbb)$ bundle $P\to Y^{reg}$,
    isomorphic (as projective bundles) to $\PBbb(\Fcal)$ on $Y\setminus S$. 
    Let $\Lcal\to P$ be the pull-back of $\Ocal_{\PBbb(\Fcal)}(1)$ under
    this isomorphism. Again using the codimension of $S$, we know that
    $\Lcal$ extends as line bundle over $\pi : P\to Y^{reg}$. 
    Moreover, $\Lcal$ restricts to $\Ocal(1)$ on each fiber of $P$. Hence,
    $\pi_\ast\Lcal =: \widetilde\Fcal^\ast$ is a vector bundle with a
    projectively flat connection. Since $\widetilde \Fcal^\ast\simeq
    \Fcal^\ast$ on $Y\setminus S$, and $\Fcal$ is reflexive,  we have 
    $\widetilde \Fcal\simeq \Fcal$ on $Y^{reg}$. This completes the proof.
\end{proof}

These arguments can be further used to  improve on several previous
algebro-geometric results
corresponding to the case of the normalized tautological class of 
$\F$ being nef, and $\F$ pseudo-effective, which we will give definitions below. We start with $Y\subset \P^N$, a projective normal subvariety 
that is
smooth in codimension $2$.
Let $\Fcal$ be a reflexive sheaf over $Y$. Denote
$\Fcal^0=\Fcal|_{Y\setminus (Y^s \cup \Sing(\Fcal))}$. Now we define the \emph{normalized tautological class} by
$$
\zeta_\F:=c_1(\O_{\PBbb(\Fcal^0)}(1)\otimes \pi^*(\det(\F^0)^{-1}))\in
H^2(\PBbb(\Fcal^0), \ZBbb)\ ,
$$
where $\pi: \PBbb(\Fcal^0) \rightarrow Y\setminus (Y^s \cup \Sing(\Fcal))$
denotes the natural projection. Here we used the algebro-geometric
convention $\PBbb(\Fcal^0):=\text{Proj}(\oplus_{k\geq 0} \text{Sym}^k \Fcal^0)$.
\begin{defi}
Given $\Fcal$ as above
\begin{itemize}
\item $\Fcal$ is called nef if $\zeta_{\F}|_{\pi^{-1}(V)}$ is nef for any smooth projective subvarieties $V\subset Y\setminus (Y^s \cup \Sing(\Fcal)$ of dimension $1,2$. 

\item \emph{pseudo-effective} if for any $c>0$ there exists $j\in \mathbb
    N$ so that for $i>cj$, 
$$
H^0(X, S^{[i]}\F \otimes \O(jH)) \neq 0
$$
        for some (and hence any) ample Cartier divisor $H$. Here $S^{[i]}(\F):=(\text{Sym}^i(\F))^{**}$.
\end{itemize}
\end{defi}

Following the argument as \cite[Thm.\ 4.1]{Nakayama:04}, we will prove the following generalization of Corollary \ref{cor:bogomolov}.
\begin{thm}[{\sc Normalized tautological class is nef}]\label{Theorem: nef}
Let $\F$ be a reflexive sheaf over $Y$. The following are equivalent:
\begin{enumerate}
\item $\F$ is locally free over $Y^{reg}$ and $\zeta_\F$ is nef.

\item $\F$ is semistable for some ample divisor $A$ (not necessarily the one from the embedding $Y\subset \P^N$) and
$$
        (2rc_2(\F)-(r-1)c_1(\F)^2)\cdot [A]^{n-2} = 0\ ;
$$
\item $\F$ admits a filtration:
$
0\subset \F_1 \subset \cdots\subset \F_m=\F
$,
where $\F_i/\F_{i-1}$ are projectively flat over $Y^{reg}$ and $\mu(\F_i/\F_{i-1})=\mu(\F)$.
\end{enumerate}
\end{thm}
\begin{proof}
Given any sheaf $\Hcal$, we denote 
$$\Delta(\Hcal):=2rc_2(\Hcal)-(r-1)c_1^2(\Hcal)\ .$$
We assume (1) holds and prove (2). 
    By definition, we know for any smooth curve $C$ lying in $Y^{reg}$ as
    the intersections of hypersurfaces in $|\Ocal(A)|$, $\tau_{\Fcal^0}|_C$
    is  nef, which by the well-known result
    is equivalent to $\Fcal^0|_C$ being semistable. 
    This implies $\Fcal$ is semistable.
    Choose $S$ to be a smooth projective surface in $Y \setminus Y^s$
    as an intersection of smooth hypersurfaces in $|\O(A)|$
    so that $\Fcal|_S$ is locally free and semistable. 
    Since $\tau_{\Fcal}|_S$ is nef by assumption, by the well-known result
    in the smooth case, 
$$
    \Delta(\Fcal)\cdot [A]^{n-2}=
    \Delta(\Fcal\bigr|_S)=0\ .
$$   
Suppose (2) holds. We prove (3) by induction on rank,
    where the rank $1$ case follows from Corollary \ref{cor:bogomolov}. Assume the statement holds for rank smaller than $\rank(\Fcal)$. We can choose $\Ecal_1\subset \Fcal$ to be a saturated stable subsheaf with $\mu(\Ecal_1)=\mu(\Fcal)$. Then we have the following exact sequence 
    $$
0\rightarrow \Ecal_1 \rightarrow \Fcal \rightarrow \Gcal \rightarrow 0
    $$
    where 
    $\Ecal_1$ is projectively flat over $Y^{reg}$ by Corollary
    \ref{cor:bogomolov}, and 
    $\Gcal$ is torsion-free and semistable with $\mu(\Gcal)=\mu(\Fcal)$. The statement follows from 
    \begin{itemize}
        \item[(a)] $\Delta(\Ecal_1)\cdot [A]^{n-2}=\Delta(\Gcal)\cdot
            [A]^{n-2}=\Delta(\Gcal^{**})\cdot [A]^{n-2}$=0;
        \item[(b)] $\Gcal^{**}=\Gcal$ over $Y^{reg}$.
    \end{itemize}    
    Indeed, we can apply the induction to $\Ecal_1$ and $\Gcal^{**}$ to get
    such filtrations as in (3), and these  naturally give  a filtration for $\Fcal$ over $Y^{reg}$.
    We first prove $(a)$. Choose  a smooth projective surface $S\subset 
    Y^{reg}$ given by the intersection of smooth hypersurfaces in
    $|\O_Y(A)|$, and  so that $\Ecal_1|_S$  and 
    $\Gcal|_S$ are torsion-free and semistable, and $\Fcal|_S$ is locally
    free and  semistable.  We have the exact sequence  
    $$
0\lra\Ecal_1|_S \lra\Fcal|_S \lra\Gcal|_S \lra 0\ .
    $$
    By \cite[Corollary 7.3.2]{HuybrechtsLehn:10}, we know  that
    $$
 \Delta(\E_1)\cdot[S]=\Delta(\Fcal)\cdot[S]=\Delta(\Gcal)\cdot[S]=0\ .
    $$
Together with the Bogomolov inequality for $(\Gcal|_S)^{**}$, this implies 
$\Delta((\Gcal|_S)^{**})=0$. 
Thus $\Gcal|_S=(\Gcal|_S)^{**}$ and $\Delta(\Gcal^{**})\cdot [S]=0$. Note
    this argument in particular implies $\supp(\Gcal^{**}/\Gcal)$ has
    codimension at least $3$. 
    But then the argument in \cite[Prop.\ 2.3]{SibleyWentworth:15} shows
    that $\Gcal=\Gcal^{\ast\ast}$ on $Y^{reg}$, and this proves (b).
    Now suppose (3) holds and we prove (1). Obviously, $\Fcal$ is locally
    free over $Y^{reg}$. For any smooth projective subvarieties $V \subset
    Y^{reg}$ of dimension $1$ or $2$, we know $\F|_{V}$ admits a
    filtration so that the graded factors are projectively flat over $V$.
    The conclusion thus follows from the known results in the smooth case.
    \end{proof}

We also have
\begin{thm}[{\sc Pseudo-effective}]
Suppose $\F$ is a stable reflexive sheaf over $Y$ so that 
$
    c_1(\F)\cdot [A]^{n-1}=0
$, 
for some ample Cartier divisor $A$ and $S^{[l]}\F$ are all indecomposable. If $\F$ is pseudo-effective, then $\F|_{Y^{reg}}$ is flat.  
\end{thm}

\begin{proof}
    By Theorem \ref{Singular DUY}, $\Fcal$ carries an admissible HE metric;
    hence, so does $S^{[l]}(\F)$.
    By Corollary \ref{cor:polystability} and the assumption of
    indecomposability, 
$S^{[l]}(\F)$ must be stable.
    By \cite[Thm.\ 1.1]{HoringPeternell:19} we know  that
$$
    c_1(\F)\cdot [H]^{n-1}=c_2(\F)\cdot [H]^{n-2}=0\ ,
$$
and 
by Theorem \ref{Theorem: nef}, we conclude that $\F|_{Y^{reg}}$ is flat. 
\end{proof}

\begin{rmk}
In \cite{HoringPeternell:19}, $S^{[l]}\F$ is assumed to be stable, and it is
    concluded that $c_1(\F)^2\cdot H^{n-2}=0$ and $c_2(\F)\cdot H^{n-2}=0$. Then by
    assuming $Y$ is klt and using \cite{GKP:16a}, the authors can pass to a finite
    Galois cover $\nu:\widetilde{Y}\rightarrow Y$  which is \'etale in
    codimension $1$,  and thence conclude that 
    the reflexive pullback of $\F$ is a numerically flat bundle.
\end{rmk}

\section{Degenerations of semistable bundles}\label{Section:degeneration}
In this section, we fix $\Xcal$ to be a subvariety of $\P^N \times \P^1$ so that 
$$p:\Xcal \lra \P^1=\C\cup \{\infty\}$$
is a family with $p^{-1}(t)$ smooth for $t\neq 0$ and $p^{-1}(0)$ normal. 
Denote $X_t=p^{-1}(t)$. 
Let $\{\Ecal_t\}_{t\neq 0}$ be a family of semistable bundles. 
We say that  $\{\Ecal_t\}_{t\neq0}$ is an \emph{algebraic family} if we can 
find a coherent sheaf $\E\to\Xcal$ so that $\E|_{X_t}\cong\E_t$ for all
$t\neq 0$.
In the following we shall always consider such algebraic families.
\subsection{The algebraic side}
We use the hyperplane section $D\subset \PBbb^N$ to define stability of
sheaves on the  fibers  of $\pi:\Xcal\to\PBbb^1$.
\begin{defi}\label{defi4.1}
A coherent sheaf $\E\to\Xcal$ is called an algebraic degeneration of 
    $\{\E_t\}_{t\neq0}$ if $\Ecal$ is flat over $\PBbb^1$, $\E|_{X_t}\cong\E_t$ for $t\neq 0$, and
    $\E_0:=\E|_{X_0}$ is torsion-free. $\E$ is called a semistable degeneration of $\{\E_t\}_{t\neq0}$ if furthermore $\E_0$ is torsion-free and semistable. 
\end{defi}
We need the following simple observation following from the flatness
assumption.
\begin{lem}\label{Lemma: Same determinant}
For different semistable degenerations $\E'$ and $\E$ of the same family $\{\E_t\}_{t\neq0}$, 
$\det \E'_0 \cong \det \E_0$.
\end{lem}
To state the main results of this section, we make the following
definition. Given a semistable sheaf $\F$ over a projective normal variety
$Y\subset \P^N$ with a fixed polarization. Let $\Gr(\F)$ be the
graded sheaf associated to a Jordan-H\"older filtration of $\F$.
Assume that $Y$ is smooth in codimension $2$. Let
$\Tcal=\Gr(\F)^{**}/\Gr(\F)$. Then $\dim \supp(\Tcal) \leq n-2$. 
\begin{defi}[Lemma]\label{Definition:algebraic blow-up locus}
The algebraic blow-up locus of $\F$ is defined as  
\begin{equation}
\mathcal C(\F)=\sum m_k^{alg} \Sigma_k \ ,
\end{equation}
where $\Sigma_k$ are irreducible codimension $2$ subvarieties of $Y$ and $m_k^{alg}$ is equal to $h^0(D, \mathcal{T}|_{D})$ where $D$ is a generic transverse slice of $\Sigma_k$. In particular, $\mathcal{C} (\E)$ is a finite sum.
\end{defi}

We refer the readers for detailed proof in Section $2.5$ in \cite{GSTW:18},
which is done for smooth projective varieties but can be seen to work here.

\subsubsection{Hecke transform and existence of semistable degenerations}
In this section, we will prove the first part of Theorem \ref{Theorem: algebraic degeneration}, i.e.\ a
semistable degeneration always exists which is a slight generalization of
\cite{Langton:75} using the same argument outlined in \cite[p.\ 59]{HuybrechtsLehn:10}, where  the
completeness of the moduli space of semistable sheaves over a fixed
projective manifold is studied. For the interest of the reader, the extension of a sheaf
across a negative divisor is also studied in \cite{ChenSun:20b}. 

We fix $\E$ to be any reflexive sheaf over $\Xcal$ so that  
$\E|_{X_t}=\E_t$ for $t\neq 0$.  This always exists by assumption.  Since $\E$ is torsion-free, we know $\E$ always defines a flat family on $\Xcal$ over $\P^1$.  We start with the following  observation which generalizes \cite[Lemma 3.23]{ChenSun:20a} by the same argument.   

\begin{lem}\label{lem3.4}
$\E_0$ is torsion-free.  
\end{lem}

\begin{proof}
Fix any $x\in X_0$. By the  assumptions, we can represent $\mathcal E$ near $x$ as 
\begin{equation}\label{equation4.1}
0\lra \mathcal E|_U \lra \O_U^{\oplus n_1}
    \xrightarrow{\phi} \O_U^{\oplus n_2}\ .
\end{equation}
where $U$ is an open neighborhood of $x$ in $\Xcal$. Restricting this exact sequence to $D=U \cap X_0$, one gets an exact sequence 
$$
\text{Tor}_1(\mathcal{O}_D, \text{Im}(\phi)) \lra \E|_D \lra
    \O^{\oplus n_1}_D\ .
$$
Since $\text{Im}(\phi)$ is torsion-free, by 
    using the natural resolution for $\mathcal{O}_D$ we have 
$
\text{Tor}_1(\mathcal{O}_D, \text{Im}(\phi))=0
$.
 This implies 
$\E|_D \hookrightarrow
    \O^{\oplus n_1}_D$. In particular, $\E|_D$ is torsion-free. 
\end{proof}

Given any saturated subsheaf $\underline{\mathcal S}\subset \E_0$, i.e. $\E_0/\underline{\mathcal S}$ is torsion-free. Let $\E'$ be the sheaf given by the following exact sequence 
$$
0\lra \E'\lra \E\lra (\iota_{0})_* (\E_0/\underline{\mathcal S}) \lra 0 \ .
$$
\begin{defi}
$\E'$ is called a Hecke transform of $\E$ along $\underline{\mathcal S}$.  
\end{defi} 

By repeating exactly the same argument as \cite[Lemma 2.4]{ChenSun:20b}, we have 
\begin{lem}\label{lem3.6}
$\E'$ is reflexive and $\E'_{0}$ is torsion-free. Furthermore, there exists a short exact sequence 
$$
0 \lra\E_0/\underline {\mathcal S} \lra \E'_0 \lra \underline{\mathcal S}
    \lra 0\ .
$$
\end{lem}
Now we start with $\E$ as above. We define a sequence of $\Ecal^k$ inductively by defining $\Ecal^1=\Ecal$ and $\Ecal^{k+1}$ to be the Hecke transform of $\Ecal^{k}$ along the maximal destabilizing subsheaf $\underline{\F}_{k}$ of $\E^k_0$. For each $\E^{k}$, we can associate it with a nonnegative number
$$
\beta_k:=\mu(\underline{\F}_k)-\mu(\E^{k}_0)=\mu(\underline{\F}_k)-\mu(\Ecal_0)
$$
which measures how far $\E^{k}$ is from being semistable. Note that $\beta_k=0$ if and only if 
$\E^{k}$ is a semistable degeneration of $\{\E_t\}_{t\neq 0}$. In the following, we denote 
$$
\underline{\Gcal}_k=\Ecal^{k}_0/\underline{\F}_k.
$$
By definition, we have the following exact sequence 
\begin{equation}
0\lra\underline{\F}_{k-1}\lra\Ecal_0^{k-1} \lra\underline{\Gcal}_{k-1} \lra
    0\ ,
\end{equation}
and by Lemma \ref{lem3.6}, we have 
\begin{equation}\label{equation4.6}
0\lra\underline{\G}_{k-1} \lra\Ecal^{k}_0 \lra\underline{\Fcal}_{k-1}
    \lra0\ .
\end{equation}

\begin{lem}\label{lem4.7}
 $\beta_k$ decreases and for $k$ large, the following hold:
    \begin{itemize}
        \item $\beta_k=\beta_{k-1}$; 
        \item there are natural isomorphisms $\underline{\F}_{k} \cong \underline{\F}_{k-1}$ and $\underline{\Gcal}_k\cong\underline{\Gcal}_{k-1}$. In particular, $\E_0^{k}=\underline{\Fcal}_k \oplus \underline{\Gcal}_k$.
    \end{itemize}
\end{lem}

\begin{proof}
It follows from \eqref{equation4.6} that $\mu(\underline{\F}_{k}) \leq \mu(\underline{\F}_{k-1})$
    since through the last map in the exact sequence, one has a map
    $\underline{\F}_{k} \rightarrow \underline{\F}_{k-1}$. Indeed, this
    follows from $\underline{\Fcal}_{k-1}$ being semistable if the map is
    nontrivial. Otherwise, we have a nontrivial map $\underline{\Fcal}_{k}
    \rightarrow \underline{\Gcal}_{k-1}$ and the maximal destabilizing sheaf of $\underline{\Gcal}_{k-1}$ has slope strictly smaller than $\underline{\Fcal}_{k-1}$.   In particular, we know $\beta_k$ decreases as $k$ increases, thus must stabilize for $k$ large since it is nonnegative. Now we assume the equality holds $\beta_k=\beta_{k-1}$ i.e. $\mu(\underline{\Fcal}_k)=\mu(\underline{\Fcal}_{k-1})$. Then we must have an injective map $\underline{\Fcal}_{k} \rightarrow \underline{\Fcal}_{k-1}$. Otherwise, we have an exact sequence
    $$
    0\lra\underline{\Gcal}'_k \lra\underline{\Fcal}_{k} \lra\underline{\Fcal}_{k-1}
    $$
    where $\underline{\Gcal}'_k \subset \underline{\Gcal}_k$. This implies $\mu(\underline{\Fcal}_k)< \mu(\underline{\Fcal}_{k-1})$, which is a contradiction. Given this, if the equality holds, we have the following exact sequence  by the sequence (\ref{equation4.6}) 
    $$
    0\lra\underline{\G}_{k-1}
    \lra\underline{\Gcal}_k=\Ecal^{k}_0/\underline{\Fcal}_k
    \lra\underline{\Fcal}_{k-1}/\underline{\Fcal}_k \lra 0 \ ,
    $$
    thus an injective map 
    $
    \underline{\Gcal}_{k-1} \rightarrow \underline{\Gcal}_k
    $
    for $k$ large. Since
    $\mu(\underline{\Gcal}_{k-1})=\mu(\underline{\Gcal}_{k})$, the
    inclusion map must induce an isomorphism in codimension $1$, thus $\underline{\Gcal}_k^{**}$ are the same for $k$ large. This then implies $\underline{\Gcal}_{k-1} \cong \underline{\Gcal}_{k}$ for $k$ large.  Thus $\underline{\Fcal}_{k} \cong\underline{\Fcal}_{k-1}$.
\end{proof}

 Now we are ready to prove the main results by following the argument
 in \cite[p.\ 59]{HuybrechtsLehn:10}.

\begin{prop}
For $k$ larger, $\E^{k}$
    is a semistable
    degeneration of $\{\E_t\}_{t}$ for $k$ sufficiently large. 
\end{prop}

\begin{proof}
    Otherwise, by replacing $\Ecal$ with $\Ecal^{k_0}$ for some $k_0$ large, we assume all the conditions in Lemma \ref{lem4.7} hold below. Denote $\underline{\F}=\underline{\F}_k$ and $\underline{\Gcal}=\underline{\Gcal}_k$. Let $R$ be the complete DVR associated to the point  $0\in
    \P^1$, and  let $\Xcal_R$ be the fiber over $\text{Spec}(R)$. Denote $z$ to be the generator of the maximal ideal of $R$ and the fibration still as $p: \Xcal_R \rightarrow \text{Spec}(R)$. Abusing notation, we continue to denote by
    $\E^{k}$ the restriction of $\E^{k}$ to $\Xcal_R$. Then $\E^{k}$
    is semistable on the generic fiber of $\Xcal_R\to\text{Spec}(R)$. Let $\Qcal^k=\Ecal/\Ecal^k$. Then $\Qcal^k|_{X_0}\cong \Gcal$ and by definitions there are short exact sequence 
    $$
    0\lra\iota_*\underline{\Gcal} \lra\Qcal^{k+1} \lra\Qcal^k \lra 0 \ ,
    $$
where $\iota: X_0 \rightarrow \Xcal_R$ denote the natural embedding. From
    this, we conclude that $\Qcal^k$ is an $R/(z^k)$-flat quotient of
    $\Ecal/(z^k)\Ecal$ by \cite[Lemma 2.1.3]{HuybrechtsLehn:10}. In
    particular,  $\Spec{R/(z^k)} \in \imag (\pi)$ where $\pi:
    \Quot_{\Xcal_R/R}(\Ecal, P(\underline{\Gcal})) \rightarrow \Spec{R}$,
    thus $\pi$ must be surjective and this gives a flat quotient  $\Ecal
    \rightarrow \Gcal$ in the quot scheme $\Quot_{\Xcal_R/R}(\Ecal,
    P(\underline{\Gcal}))$ which destabilizes the sheaf on the generic fiber. Contradiction.
\end{proof}

\subsubsection{Uniqueness of $Gr^{**}$}\label{Double dual}
The uniqueness of the double dual of the limiting sheaf in  Theorem \ref{Theorem: algebraic degeneration}
 is a direct corollary of the following two lemmas.

\begin{lem}\label{lem5.8}
Given any two semistable degenerations $\E$ and $\E'$ of $\{\E_t\}_t$, for $l$ large, for generic family of curves $\mathcal{Y} \subset \Xcal$ defined by $\mathcal{O}(l)$, denote $Y_t=p^{-1}(t) \cap \mathcal{Y}$, then $\E|_{Y_0}$ and $\E'|_{Y_0}$ are both semistable vector bundles for $t\neq 0$. In particular, $\E|_{Y_0}$ and $\E'|_{Y_0}$ are $S$-equivalent. 
\end{lem}
\begin{proof}
The first part is a direct corollary of the MR restriction theorem. Now we also know $\E|_{\mathcal Y_t}$ and $\E'|_{\mathcal Y_t}$ is
    semistable for $t\neq 0$. The second conclusion follows from the
   geometry of a relative version of  the moduli space of semistable bundles
    over curves (see  \cite[Thm.\ 4.3.7]{HuybrechtsLehn:10}).
\end{proof}

Also we need the following well-known lemma.
\begin{lem}\label{lem5.9}
$\Gr(\E_0)^{**}=\Gr(\E'_0 )^{**}$  if and only if $\E_0|_{Y}$ and
    $\E'_0|_{Y}$ are $S$-equivalent for generic curves $Y$ given as a 
    complete intersection of hypersurfaces in 
    $\P(H^0(X_0, \O(l))^{*})$, for $l$ large.
\end{lem}

\subsubsection{Uniqueness of algebraic blow-up locus when $\codim X_0^s \geq 3$}
In this section, we assume $n\geq 3$.  Given $\E$ and $\E'$ two semistable
degenerations of $\{\E_t\}_t$, assume $\mathcal C(\E) \neq \mathcal C(\E')$,
similar to  \cite[p.\ 64]{GrebToma:17}, we have  
\begin{prop}\label{prop5.1}
By restricting to a generic family of smooth surfaces transverse to  $\mathcal{C}(\E') \cup \mathcal C(\E)$, we get two different semistable degenerations of the corresponding restriction of $\{\E_t\}_t$ in terms of the algebraic blow-locus.
\end{prop}

Given this, in the following, we can assume $\Xcal$ to be a family of smooth projective surfaces and two semistable degenerations  $\E$ and $\E'$  of $\{\Ecal_t\}$ so that $\Ccal(\Ecal_0)\neq \Ccal(\Ecal_0')$. To show uniqueness of bubbling sets, we first briefly recall
how the moduli space of semistable sheaves with given determinants of
numerical classes over $X_0=p^{-1}(0)$ is constructed. We refer readers to
\cite{HuybrechtsLehn:10} for more details. Using the fact that the set of
isomorphism classes  of semistable sheaves is bounded, 
and hence $m$-regular for some integer $m$, we can put those sheaves in a fixed Quot scheme. More precisely, denote 
$$\mathcal{H}:=\O_{X_0}(-m)^{\oplus P(m)}\ ,$$ 
and let 
$R^{\mu ss} \subset \Quot(\mathcal H, P)$ 
be the locally closed subscheme of all quotients $[q:\mathcal{H}
\rightarrow \F]$ such that $\mathcal{\F}$ is torsion-free
$\mu$-semistable with rank equal to $r$, determinant equal to $Q$, and same numerical
classes. Furthermore, $q$ induces an isomorphism between $V$ and
$H^0(\F(m))$. The group $\SL(V)$ acts on $R^{\mu ss}$ naturally. Now we consider the universal quotient 
$\tilde{q}: \mathcal{O}_{R^{\mu ss}}\otimes \mathcal{H}
\lra\widetilde{\F}$,
and the line bundle 
$
\mathcal{N}:=\lambda_{\widetilde{\F}}(u_1(c))
$, 
where $u_1(c)$ is certain sheaf associated to $\O(1)$ and $\O_x$ for some
point $x\in X_0$. For some positive integer $\nu>0$, $\mathcal{N}^{v}$ is
generated by $\SL(V)$-invariant global sections. Furthermore, there is an integer $N>0$ so that 
$$
\bigoplus_{l\geq 0} H^0(R^{\mu ss}, \mathcal{N}^{lN})^{\SL(V)}
$$
is a finitely generated graded ring. Now define 
$$
M^{\mu ss}:=\text{Proj}\biggl(\bigoplus_{k\geq 0} H^0(R^{\mu ss}, \mathcal{N}^{k
N})^{\SL(V)}\biggr)\ ,
$$
where $H^0(R^{\mu ss}, \mathcal{N}^{k N})^{\SL(V)}$ denotes the space of
sections in $H^0(R^{\mu ss}, \mathcal{N}^{k N})$ invariant under
the action of $\SL(V)$. 
Also, we have a map 
\begin{equation}
\Phi: R^{\mu ss} \lra M^{\mu ss}.
\end{equation}
which collapses the $\SL(V)$ orbits into points. Furthermore, the image of two points $[q_0: \mathcal{H} \rightarrow \E_0]$ and $[q_0': \mathcal{H} \rightarrow \E'_0]$ are the same in $M^{\mu ss}$ if and only if 
$$
\Gr(\E_0)^{**}=\Gr(\E_0')^{**}
\ ,\ \text{and}\ 
\mathcal{C}(\E_0)=\mathcal{C}'(\E_0')\ .
$$
We want now to do such constructions for the family $\Xcal$. Denote 
$$
\mathcal H_\Xcal= \underline{\C}^{\oplus P(m)}\otimes_{\O_{\Xcal}}
\O_{\Xcal}(-k)\ .
$$
Now we define $R^{\mu ss}_{\Xcal}\subset \Quot(\mathcal{H}_\Xcal, P)$ to be the space of flat families of torsion-free semistable sheaves over $\Xcal$ with the same numerical classes and Hilbert polynomial $P$ as the sheaves in $M^{\mu ss}$. By repeating the construction above for a family, we have a well-defined object 
$$
\mathcal M^{\mu ss}_\Xcal:=\text{Proj}\biggl(\bigoplus_{k\geq 0} H^0(R^{\mu
ss}_\Xcal, \mathcal{N}_\Xcal^{k N})^{\SL(V)}\biggr)\ ,
$$
which admits a natural rational map 
$
\Phi_{\Xcal}: R^{\mu ss}_\Xcal \dashrightarrow \mathcal{M}^{\mu ss}_\Xcal
$
over $\P^1$.
\begin{prop}\label{prop4.12}
For $N$ large, near the central fiber, $\Phi_{\Xcal}$ is well-defined and restricts to $\Phi$ on the central fiber.
\end{prop}

\begin{proof}
This follows directly from the construction actually. Indeed, the sections
    used to construct $M^{\mu ss}$ are pulled back from the moduli space of
    semistable sheaves over a high degree generic curves. In our case,
    since the family is flat, we can always assume the high degree curve in
    the central fiber fits into a flat family. Now the conclusion follows
    from  \cite[Thm.\ 4.3.7]{HuybrechtsLehn:10},  which formulates 
    the relative version of moduli space of semistable bundles over curves, and 
    this provides enough sections.
\end{proof}

Combined with Proposition \ref{prop5.1}, this implies the uniqueness of the bubbling set in Theorem \ref{Theorem: algebraic degeneration}. Indeed, this follows from that $[\Ecal_0]$ and $[\Ecal_0']$ give the same point in $M^{\mu ss}$ by continuity. In particular, $\Ccal(\Ecal_0)=\Ccal(\Ecal_0')$ which is a contradiction to our assumption. We emphasize here by assumption, we are working over a family of smooth surface.

\subsection{The analytic side} \label{analytic side} 
Now we prove Theorem \ref{Theorem: analytic degeneration}. Recall we assume $\E_t\to X_t$ 
are stable vector bundles for $t\neq 0$,  and we want to
study the analytic degeneration of the corresponding HYM connections with respect to the natural K\"ahler metric $\omega_t=\omega_{FS}|_{X_t}$ on $X_t$ induced from the embedding $\Xcal \rightarrow \PBbb_1 \times \PBbb^N$.

As for this, we always fix a sequence $\Ecal_i:=\Ecal_{t_i}$ where $t_i \rightarrow
0$ as $i\rightarrow \infty$.  Let $A_i$ be the admissible
HYM connection on $\E_{t_i}$ over $(X_{i}, \omega_i):=(\Xcal_{t_i}, \omega_{t_i})$.
By passing to a subsequence, up to gauge transforms,  $A_i$ converges to an admissible HYM connection $A_\infty$ away from a complex codimension $2$ subvariety $\Sigma \subset X_0^{reg}$. More precisely, take a precompact exhaustion of $X_0^{reg}$ as 
$
U_1 \subset U_2 \subset
\cdots\subset X_0^{reg}
$. 
Over each $U_j$, by fixing diffeomorphisms of the base, we can assume $\omega_i$ and the complex structure $J_i$ of $X_i$ are all defined on $U_j$ which converge to the complex structure and the K\"ahler metric of $X_0^{reg}\cap U_j$. Now we can apply Uhlenbeck compactness to get a limit over $U_j$ and a diagonal sequence argument gives the limiting connection $A_\infty$ as well as the bubbling set $$
\Sigma=\{x\in X_0^{reg}: \lim_{r\rightarrow 0} \liminf_{i} r^{4-2n}\int_{B_r(x)} |F_{A_i}|^2>0\}
$$ 
which is a subvariety of $X_0^{reg}$ and decomposes as $\Sigma=\Sigma_b \cup \Sing(A_\infty)$,
where $\Sigma_b$ is of pure codimension $2$ (cf.\ \cite{Tian:00}). As a sequence of Radon measures over $Y^{reg}$, 
$$
\mu_i:=|F_{A_i}|^2\dVol_i \rightharpoonup \mu_\infty=|F_{A_\infty}|^2 \dVol+\nu
\ ,
$$
where $\nu|_{Y^{reg}}=\sum_k 8\pi^2 m_k \Sigma_k$ and $\Sigma_k$ denotes
the pure codimension $2$ components of $\Sigma$. Here $m_k\in \mathbb{Z}$ is called the analytic multiplicity of $\Sigma_k$. In the following, abusing notation, we will use $\Sigma_b=\sum_k m_k \Sigma_k$ to include the multiplicity information. 
 Let $\E_\infty$ be the holomorphic vector bundle defined by $A_\infty$ over $X_0^{reg}$.

\subsubsection{Proof of Theorem \ref{Theorem: analytic degeneration} (I)}
Fix $k$ large enough and a sequence $\{s_i\}_j$ of holomorphic sections of $H^0(X_{i}, \E_i(k))$ with $L^2$ norms normalized to be $1$. 

\begin{lem}\label{lp}
There exists a constant $C$ so that 
$$\|s_i\|_{W^{1,2}}<C\ ,\ \|s_i\|_{L^\infty}<C\ .$$ 
In particular, by passing to a subsequence, $s_i$ converges locally smoothly to a nontrivial holomorphic section $s_\infty$  away from $\Sigma$ satisfying $\|s_\infty\|_{L^2}=1$ and 
$$\|s_\infty\|_{W^{1,2}}\leq C\ ,\  \|s_\infty\|_{L^\infty}\leq C\ .$$ 
\end{lem}

\begin{proof}
The $W^{1,2}$ bound follows from integration of the equation
$$
\Delta |s_i|^2= |\nabla s_i|^2-\mu |s_i|^2.
$$
Since we have a uniform Sobolev constant for $X_{t_i}$, which is due to the
    fact that $X_{t_i}$ is a family of minimal submanifold in $\P^N$ (see
    \cite{Simon:83}), the uniform $L^\infty$ bound follows from Moser
    iteration. Note here  that the $X_{t_i}$ are all smooth.
\end{proof}

Fix $k$ large enough. Choose an orthonormal basis $\{s_i^j\}_j$ for $H^0(X_i, \E_{i}(k))$.  By passing to a subsequence, we can assume $s_i^j$ converges to $s_\infty^{j}$. Let $\F^{k}_\infty$ be the image sheaf of the map
$
q_\infty: \O(-k)^{\oplus N_k} \to\E_\infty 
$,
where $q_\infty=(s_\infty^1, \cdots, s_\infty^{N_k})$. Using the argument
in  \cite[p.\ 54]{GSTW:18}, we have 
\begin{lem}
    $\F^{k}_\infty \subset \F^{k+1}_\infty$. Furthermore, over any fixed compact subset $K \subset Y^{reg}$, the equality holds for $k$ large .
\end{lem}

\begin{prop}\label{prop3.17}
    $\E_\infty$ can be extended to  a coherent reflexive  sheaf
    $\overline{\Ecal_\infty}\to X_0$. 
\end{prop}

\begin{proof}
It suffices to prove that for any $r\leq \rank(\E_\infty)$, there exists a
    \emph{saturated} subsheaf $\G_r\subset \E_\infty$, so that $\G_r$
    extends to a reflexive 
    coherent sheaf on $Y$,  and $\rank \G_r=r$. Then   for $r=\rank(\E_\infty)$, 
    we must have $\G_r=\Ecal_\infty$ on $X_0^{reg}$, and the result follows. 
    We prove the existence of $\G_r$ by induction on $r$.

    \medskip

   \noindent 
    $\bf r= 1$: Take a nonzero section  $\sigma\in
    H^0(X_0^{reg},\E_\infty(k))$. 
    Then there is a Weil divisor
    $Z\subset X_0^{reg}$ so that the image  of the injection 
$$
    0\lra \Ocal_{X_0^{reg}}(Z)\xrightarrow[\hspace*{.75cm}]{\sigma} \E_{\infty}(k)
$$
    defines a saturated rank $1$ subsheaf $\G_1\subset\E_\infty$. 
    By the
        Remmert-Stein
        extension theorem we know that
        $Z$ can be extended to  a Weil divisor
        over $X_0$, since $X_0$ is normal. In particular, this tells us 
       that  $\G_1$ can be extended to  a coherent sheaf 
        which admits a map to $\E_\infty$ away from $X_0^{s}$.
        This proves the result for $r=1$.

        \medskip

        \noindent
        $\bf r\geq 1$:
Assume the statement has been proved for $\rank=r-1$, and let $\G_{r-1}$ be a 
    saturated subsheaf of $\E_\infty$ which can be extended to be 
    a coherent sheaf over $X_0$. Now we look at $\E_{\infty}/\G_{r-1}$. By the 
    asymptotic Riemann-Roch theorem  (see Corollary \ref{ARR}),
    $\E_\infty(k)/\mathcal{G}_{r-1}(k)$ admits a section which 
    generates a rank $1$ subsheaf of $\E_\infty(k)/\mathcal{G}_{r-1}(k)$ for $k$ 
    large. Otherwise, $\mathcal{G}_{r-1}(k)$ admits $P(k)$ sections for $k$ large, 
    which is a contradiction. We pick such a section $\sigma$ of $\E_\infty(k)/\mathcal{G}_{r-1}(k)$. 
    As above, we can assume $\sigma$ gives a rank $1$ saturated coherent
    subsheaf  of $\E_\infty/\mathcal{G}_{r-1}$. Let $\mathcal{G}_{r}\subset
    \E_\infty$ be the sheaf which fits into the following exact sequence
    over $X_0^{reg}$ as 
$$
0\lra \mathcal{G}_{r-1} \lra \G_{r} \lra \mathcal{L}_1\lra 0\ .
$$
Since $\mathcal{G}_{r-1}$ and $\mathcal{L}_1$ are coherent over $X_0$, we
    know $\mathcal{G}_r$ is also coherent over $X_0$. 
\end{proof}
As Lemma \ref{Lemma: Extension of bubbling set}, we have 
\begin{cor}
  The closure  of  $\Sigma$ is a subvariety of $X_0$.
\end{cor}

\begin{proof}
    We first prove the pure codimension $2$ part can be extended. Take such a component $\Sigma_k$. Note for any compact set $K\subset Y^{reg}$,  
    $$
    \Sigma_k \cap K \leq \limsup_i \int_{X_{i}} |F_{A_i}|^2 \leq C$$
    where $C$ is independent of $K$. In particular, $\Sigma_k$ has bounded
    volume in $\PBbb^N$, thus could be extended by Bishop extension
    theorem. The extension of $\Sing(A_\infty)$ follows from
    $\Sing(A_\infty)=\Sing(\Ecal_\infty)$ while  $\Ecal_\infty$ can be extended to be a coherent sheaf over $X_0$ by Proposition \ref{prop3.17}. 
\end{proof}

 We still denote  the closure of $\Sigma$ and $\Sigma_b$ in $X_0$ by $\Sigma$ and $\Sigma_b$.
Given the extension property and
the asymptotic Riemann-Roch theorem \ref{ARR},
we have the following.
\begin{cor}\label{Corollary: subbundle of trivial ones locally}
    For $k$ large, $\F^{k}_\infty=\F^{k+1}_\infty$ and 
    $\mu(\F_\infty^{k})=\mu(\E)$. 
    In particular, the inclusion $\F_\infty^{k}\subset \E_\infty$ induces a bundle isomorphism outside a codimension $2$ subvariety.  
\end{cor}

As an easy corollary of this, we have 
\begin{prop}\label{regularity}
For any local section $s$ of $\overline{\E_\infty}$, $\log^{+}|s|^2 \in L^q_{loc}$ for any $q>1$. 
\end{prop}
\begin{proof}
    Since $\overline{\E_\infty}/\F_\infty^{k}$ is a torsion sheaf, there
    exists a holomorphic function $P$ so that $Ps\in L^2_{loc}$. Note that
    $$\log^+ |s|^2\leq \log^+ |Ps|^2+\log^+ \frac{1}{|P|^2}.$$ Since $(\log^+x)^q \leq C_q(|x|^2+1)$ for some constant $C_q$, we know 
    $(\log^+|Ps|^2)^q \in L^1_{loc}$
for any $q>1$. It remains to show $(\log^+\frac{1}{|P|^2})^q \in L^1_{loc}$. This follows from $(\log^+x)^q\leq C x^\frac{1}{N}$ for some constant
    $C=C(q,N)$ and the fact that $1/|P|^{\frac{2}{N}}$ is integrable for $N$ large, which follows from the smooth case by passing to a resolution of singularities.
\end{proof}

\begin{cor}
For any global section $s$ of $\overline{\E_\infty}$, $\log^+|s|^2 \in
    W^{1,2}$. In particular, $|s|$ is bounded. 
\end{cor}
\begin{proof}
 This follows from an easy cut-off argument. Indeed, we know 
$$\Delta \log^{+}|s|^2 \geq -\mu(\overline{\E_\infty})\ .$$ 
Take a cut-off function $\chi_\epsilon$ supported outside an $\epsilon$ neighborhood of $Y^s\cup \Sigma$ in $\P^N$ satisfying $|\nabla \chi_\epsilon|<1/\epsilon$. Now the conclusion follows by 
    multiplying the equation above by $\chi_\epsilon^2 \log^{+}|s|^2$, integrating by parts, and letting $\epsilon \rightarrow 0$. Then we get $\log^+|s|^2\in W^{1,2}$. Given this, the $L^\infty$ bound follows as Proposition \ref{prop3.2}.
\end{proof}

As a direct corollary of this, similar to Corollary \ref{cor:polystable}, we have
\begin{cor}\label{Semistablity}
$\overline{\E}_\infty$ is semi-stable with slope equal to $\mu(\F)$. 
\end{cor}

\subsubsection{Proof of Theorem \ref{Theorem: analytic degeneration} (II)}
The first part about $\Gr^{**}$ follows exactly as section \ref{Double dual} since we know the limiting sheaf is semistable. The second part follows from Corollary \ref{cor:polystability}.

\subsubsection{Proof of Theorem \ref{Theorem: analytic degeneration} (III)} 
Now we assume $X_0$ is smooth in codimension $2$. From above, we know $\E_\infty$ can be extended as a reflexive sheaf $\overline{\E_\infty}$ over $X_0$ and we also have the following natural maps
$$
q_i^k: \O(-k)^{\oplus N_k} \xrightarrow{(s_i^1, \cdots,  s_i^{N_k})}
\E_i 
$$
By passing to a subsequence, we can take a limit of $\{q_i^k\}_i$ in the corresponding relative Quot scheme, which we denote as 
$$
q^{alg, k}_\infty: \O(-k)^{\oplus N_k} \lra\F^{alg,k}.
$$ 
Similar to  \cite[Sec.\ 4.2]{GSTW:18}, we have
\begin{lem}\label{lem4.5}
$q^{alg,k}_\infty$ induces an isomorphism $\phi:\F^{alg,k} \rightarrow \F^k_\infty$ over $Y$ for $k$ large. 
\end{lem}

Denote $\F_\infty=\cup_k \F_\infty^k$ which is equal to $\F_\infty^{k}$ for any $k$ large.  Using Theorem \ref{SingularBottChern} and the argument in 
\cite[Sec.\ 4.3]{SibleyWentworth:15}, we have 
\begin{prop}\label{BubblingIsAlgebraic}
$\Sigma^{alg}_b=\mathcal{C}(\F_\infty)$. 
\end{prop}
We also observe 
\begin{lem}\label{RestrictionOfQuotScheme}
By restricting to a smooth family of surfaces $\mathcal{Y} \subset\mathcal{X}$ and taking a discrete subsequence of smooth surfaces $Y_i:=\mathcal{Y} \cap X_i$ which converges to a smooth surface $Y_\infty\subset X_0$ on the central fiber, we can assume the following 
\begin{enumerate}
\item $\E_0|_{Y_\infty}$ and $\F_\infty|_{Y_\infty}$ are torsion-free semistable with the same Hilbert polynomial;

\item $[\mathcal{H}|_{Y_i}\otimes \O_{Y_i}(-m) \rightarrow \E_i|_{Y_i}]$ converges to $[\mathcal{H}|_{Y_\infty}\otimes \O_{Y_\infty}(-m) \rightarrow \F_\infty|_{Y_\infty}]$ in the corresponding relative Quot scheme.
\end{enumerate}
In particular, $\mathcal{C}(\E_0|_{Y_\infty})=\mathcal{C}(\F_\infty|_{Y_\infty})$.
\end{lem}

\begin{proof}
The first statement follows from the restriction theorem and a
    straightforward calculation using that $\E_0$ and $\F_\infty$ have the same Hilbert polynomial. For the second statement, we can always take a limit of the sequence $[\mathcal{H}|_{Y_i}\otimes \O_{Y_i}(-m) \rightarrow \E_i|_{Y_i}]$ in the corresponding Quot scheme as $[\mathcal{H}|_{Y_\infty}\otimes \O_{Y_\infty}(-m) \rightarrow \F'_\infty]$; furthermore, $\F'_\infty$ and $\E_0|_{Y_\infty}$ have the same Hilbert polynomial. As \cite[Lemma 4.4]{GSTW:18}, there exists a surjective map $\F'_\infty \rightarrow \F_\infty|_{Y_\infty}$ which now has to be an isomorphism since they have the same Hilbert polynomial. The last statement follows from Proposition \ref{prop4.12} which implies $\F_\infty|_{Y_\infty}$ and $\E_0|_{Y_\infty}$ give the same points in $M^{\mu ss}_{Y_\infty}$ by continuity. 
\end{proof}

Similar to the uniqueness of the cycle in Theorem \ref{Theorem: algebraic degeneration}, it follows from Proposition \ref{prop5.1}, Lemma \ref{RestrictionOfQuotScheme} and Proposition \ref{BubblingIsAlgebraic} 
\begin{cor}
$\mathcal{C}(\E_0)=\mathcal C(\F_\infty)$ for any semistable degeneration $\E$ of $\{\E_t\}_t$. In particular, $\Sigma_b^{alg}=\mathcal{C}(\E_0).$
\end{cor}

\section{The Mehta-Ramanathan restriction theorem}\label{Section: MR theorem}
In this section, as an application,  
we will study the relationship between the MR restriction theorem of
semistable bundles with the degeneration of the semistable bundles through the deformation to the projective cones. 

\subsection{Deformation to projective cones} 
Fix $(X, L, \omega)$ to be a Hodge manifold, i.e. $[\omega]=c_1(L)$ for some  very ample line bundle $L$ and let $V$ be a smooth hypersurface define by a section in $H^0(X, L)$.  
We recall the  construction of the  projective cone. 
Through the embedding 
$$X \lra\P^N:=\P(H^0(X, L)^*)\ ,$$ $V$ can be realized
the intersection of $X$ with a hyperplane $\P^{N-1}$ in $\P^N$. Here $\P^N$ is endowed with the Fubini-Study metric given by the
$L^2$ metric. Using the $L^2$ metric, we can write $\P^N=\C^N\cup \P^{N-1}$
and a point in $\P^N$ will be denotes by $Z:=[Z_1, \cdots, Z_{N+1}]$. We
also denote $\widehat Z:=[Z_1, \cdots, Z_{N}, 0]$. Now we can take the affine cone $C(V)$ given by $V$ in
$\C^N$ and the projective cone $\overline{C(V)}$ is given by adding a copy
of $V$ to $C(V)$ at infinity. Denote by $\omega_0$ the induced
metric on $\overline{C(V)}\subset \P^N$. 
In general, $\overline{C(V)}$ is
not normal. We have
the following well-known lemma (see  \cite[I-Ex.\ 3.18, II-Ex.\ 5.14]{Hartshorne:77}).

\begin{lem}
The following are equivalent:
\begin{enumerate}
\item $\overline{C(V)}$ is normal; 
\item
$H^0(\P^N, \O(k)) \rightarrow H^0(V, \O(k)|_V)$ is surjective for each $k$;
\item The homogeneous coordinate ring of $X$ in $\P^N$ is an integrally closed domain. 
\end{enumerate}
\end{lem}

Let $\widehat{\mathcal X}$ 
be the blow-up of $X\times \P^1$ along $V\times \{0\}$.  Denote 
$q:\widehat{\mathcal X} \rightarrow \P^1$ as the natural projection map. We
know $q^{-1}(t)=X$ for $t\neq 0$ and $q^{-1}(0)=X\cup_{V} \overline{N_V}$
where $\overline{N_V}$ denoes the space by adding an infinity copy of $V$ to $N_V$ and $X$ are attached along
$V\subset X$ and the infinity section $V$. It
is well-known that we can blow down $\widehat\Xcal$ by contracting a copy of
$X$ in the central fiber to get $\mathcal X$, 
which now is a flat family of irreducible varieties. 
Let $p: \Xcal \rightarrow \P^1$  denote the natural projection map and
denote $X_t=p^{-1}(t)$ as before. The subtlety is that  $X_0$ might not be
normal in general, and so in order to 
 apply the degeneration results we obtained in Section \ref{Section:degeneration}
 we need to make this extra assumption.

Let $\mathcal{L}$ the line bundle obtained by first pulling back $L$ 
over $X\times \P^1,$ then pulling back to $\widehat{\mathcal{X}}$ and tensoring
it with $-\overline{N_V}$, finally pushing down to $\mathcal{X}$ through
the blow-down map.  We have the following result 
and examples from \cite{ConlonHein:14, LiChi:20}.
\begin{lem}\label{lem1.4}

\begin{enumerate}
\item There exists a natural map $\phi: \overline{C(V)} \rightarrow X_0$ which is an isomorphism away from the vertex. It is an isomorphism if and only if $H^0(X, \O(kL)) \rightarrow H^0(V, \O(kL|_V))$ is surjective for each $k$. 

\item $\Xcal$ is a flat family and there exists an embedding $\iota: \Xcal \rightarrow \P^N\times \P^1$ where $\mathcal L = \iota^* \pi_1^*\O_{\P^N}(1)$. Here $\pi_1: \P^N \times \P^1 \rightarrow \P^N$ denotes the natural projection map. 
\end{enumerate}
\end{lem}

\begin{exam}
In the construction, let $X$ be a Riemann surface of $g\geq 1$ and $V$ be a point $p$. In this case, one readily checks that $\overline{C(V)}=\P^1$ while $X_0$ has to be singular since its arithmetic genus is $g\geq 1$. 
\end{exam}

\begin{exam}
In the construction, if $X$ is Fano or a simply connected Calabi-Yau manifold, $X_0$ always coincides with $\overline{C(V)}$. 
\end{exam}

\begin{exam}
In the construction, we take $X=\P^n$, $L=\O(k)$ and $V$ to be a degree $k$
    smooth hypersurface. In this case, it is known that $V$ must be
    projectively normal (see 
    \cite[Ex.\ II.8.4]{Hartshorne:77}). Furthermore, all the conditions in Lemma \ref{lem1.4} are satisfied. The central fiber $X_0$ is normal and $X_0=\overline{C(V)}$. When $k=1$, we know $\overline{C(V)}=\P^{n}$. 
\end{exam}

We will use the following notation for the cone $\overline{C(V)}$
\begin{itemize}
\item denote the vertex of $\overline{C(V)}$ by $o\in \overline{C(V)}$;
\item let $\pi: \overline{C(V)}\setminus \{o\} \rightarrow V$ be the natural surjective map;
\item denote $\iota: \overline{C(V)} \setminus \{o\}\hookrightarrow \overline{C(V)}$ the natural embedding.
\end{itemize}
For use later, we observe the following

\begin{cor}[HE metrics on cone-type sheaves]\label{RHYM}
Let
    $\E=\iota_*\pi^*(\E|_V)$ be a polystable stable reflexive sheaf over
    $\overline{C(V)}$, then an admissible HE metric on $\E$ can be given
    by $$h:=\frac{|\widehat{Z}|^{2\mu}}{|Z|^{2\mu}} \pi^* \underline h$$
    where $\underline h$ is a HE metric on $\E|_V$ and $\mu=\mu(\E)$. In
    particular, the admissible HE metric on $\E$ restricts to the
    admissible HE metric on $\E|_V$. 
\end{cor} 

\begin{proof}
It suffices to show that 
$h$
is also an admissible HE metric on $\E$. Indeed, by definition 
$$
F_h=\mu \partial \db \ln |Z|^2\Id -\mu \partial \db \ln |\widehat Z|^2\Id+
    F_{\pi^*\underline{h}}\ .
$$ 
Since 
$$
\omega_{0}|_{C(V)}=\sqrt{-1}\partial \db \ln(1+|z|^2)=
\frac{|z|^2}{1+|z|^2} \sqrt{-1} \partial \db \ln |z|^2 + \frac{\sqrt{-1}
    \partial |z|^2 \wedge \db |z|^2}{(|z|^2+1)^2|z|^2}\ ,
$$
where $z=(\frac{Z_1}{Z_{N+1}}, \cdots, \frac{Z_N}{Z_{N+1}})$, by a direct computation
    we have
$$
\sqrt{-1} \Lambda_{\omega_0} F_{\pi^* \underline h}=\frac{|Z|^2}{|\widehat
    Z|^2} \mu \Id\ .
$$ 
Similarly,
$$
\sqrt{-1} \Lambda_{\omega_0} \partial \db \ln |\widehat
    Z|^2=\frac{|Z|^2}{|\widehat{Z}|^2} \Id\ .
$$ 
In particular, we have 
$\sqrt{-1}\Lambda_{\omega_0} F_h=\mu \Id$. Also it is obvious that the
    $L^2$ norm of $F_h$ is bounded over $\overline{C(V)}$, since it has
    quadratic blow-up when $n \geq 3$, and $\sqrt{-1}F_{\underline h}=\mu \id\omega_{FS}|_{V}$ when $n=2$. The conclusion follows.
\end{proof}

\subsection{Semistable bundles through deformation to projective cones} In
the following, we will work under the assumption that $X_0$
constructed above is normal. We will use the induced metrics from the
embedding $\Xcal \subset \P^N \times \P^1$ from $(3)$ in Lemma \ref{lem1.4}
for the deformation of the base as projective varieties. 
Notice that a semistable bundle $\Fcal\to X$ gives rise to an algebraic
family $\{\E_t=\F\}_t$. 
As a special case of Theorem \ref{Theorem: algebraic degeneration}, we have 
 
\begin{cor}\label{Theorem: general restriction}
Assume $X_0$ is normal.
    A semistable degeneration of $\{\E_t=\F\}_t$ always exists. For different semistable degenerations $\E$ and $\E'$, 
$$
\Gr(\E_0)^{**}=\Gr(\E'_0)^{**}\ ,
$$
and when $n \geq 3$,
$
\mathcal{C}(\E_0)=\mathcal{C}(\E_0')
$.
\end{cor}

Assume now that  $\F$ is stable, and let $A_t$ be the unique HYM connection on
$\E_t=\F$ with respect to the corresponding induced metric. As a special case of the discussion in Section \ref{analytic side},
by passing to a subsequence, we can assume $A_t$ converges to $A_\infty$
with blow-up locus $\Sigma_b$ subvarieties of $X_0$. Let $\E_\infty$ be the
reflexive sheaf over $X_0$ defined by $A_\infty$. As a consequence of Theorem \ref{Theorem: analytic degeneration}, we have 
\begin{cor}\label{cor5.9}
Assume $X_0$ is normal. For each limiting pair $(A_\infty, \Sigma_b)$,
\begin{itemize}
\item $\E_\infty$ can be extended to a reflexive sheaf $\overline{\E_\infty}$ over
    $\overline{C(V)}$. Furthermore, for any $q>1$; for any global section $s$, $\log^+ |s|^2\in W^{1,2} \cap L^\infty$;
    
    \item $\Gr(\overline{\E_\infty})^{**}=\Gr(\E_0)^{**}$;

\item Assume $n\geq 3$. Then
    $\Sigma_b=\mathcal C(\E_0)$ for any semistable degeneration $\E$ of $\{\E_t=\F\}$ through the deformation to projective cones and $\overline{\E_\infty}=\Gr(\E_0)^{**}$
\end{itemize}
\end{cor}

\begin{rmk}
\begin{itemize}
\item In a sense, under the given assumptions, our results produce a 
    canonical object when the restrictions of semistable bundles to 
        smooth hypersurfaces fail to be semistable. 

\item Using the example $\E=T_{\P^n}$ and $V=\P^{n-1}$, $A_\infty$ does not
    have a cone property. Indeed,  in this example $A_\infty$ will be the
        original HYM connection because the isomorphism class of $T_{\P^n}$
        is $\C^*$ invariant.  This is very different from \cite{ChenSun:19}.
\end{itemize}
\end{rmk}

By replacing $L$ with $L^k$ in the construction and choosing $V$ generic,
we know from the MR restriction theorem that $\F|_V$ is stable. In this
case, a natural semistable degeneration $\E$ can be constructed so that
$\E_0|_{\overline{C(V)} \setminus \{o\}}=\pi^*(\E|_V)$. By Corollary \ref{RHYM} and Corollary \ref{cor5.9}, we have the following analytic version of MR restriction theorem for stable bundles
\begin{cor}\label{cor6.10}
Assume $X_0=\overline{C(V)}$ is normal and that $\F|_{V}$ is stable. 
    Then $\F_\infty=\iota_*\pi^* (\F|_V)$ and $A_\infty|_{V}$ is the HYM connection on $\F|_V$.
\end{cor}
\begin{proof}
The sheaf part is clear. For the connection part, it suffices to notice
    that for the HE metric constructed on $\iota_*\pi^*(\F_V)$ in Corollary
    \ref{RHYM}, the sheaf is generated by $L^2$ sections,  as is the sheaf
    defined by the limiting HYM connection. In particular, the isomorphism between $\F_\infty$
    and $\iota_*\pi^*(\F_V)$ is uniformly bounded, thus parallel as in
    Corollary \ref{cor:polystability}.
\end{proof}

\bigskip

\noindent {\bf Data Availability Statement}. Data sharing not applicable to this article as no datasets were generated or analyzed during the current study.

\medskip

\noindent{\bf Conflict of Interest Statement}.
On behalf of all authors, the corresponding author states that there is no conflict of interest.

\bibliography{./papers}

\end{document}